\pdfoutput=1
\documentclass[11pt,a4paper]{article}
\usepackage[top=1.5545in,bottom=1.5545in,left=1.5545in, right=1.5545in]{geometry}
\usepackage{amsmath,amsfonts,amsthm,amssymb,amscd}
\usepackage{verbatim}
\usepackage{mathtools}
\usepackage{mathrsfs}
\usepackage{csquotes}
\usepackage{verbatim}
\usepackage{enumerate}
\usepackage[normalem]{ulem}
\usepackage[svgnames]{xcolor}
\usepackage{tocloft}

\usepackage[looser]{newtxtext}
\usepackage[bigdelims,varvw]{newtxmath}
\usepackage[protrusion=true, tracking=true, kerning=true,spacing=true]{microtype}

\usepackage[colorlinks=true,linkcolor=NavyBlue,
citecolor=NavyBlue,urlcolor=NavyBlue]{hyperref}
\usepackage{cleveref}
\crefname{equation}{}{}

\usepackage{subcaption}

\usepackage{tikz}
\usepackage{pgfplots}
\usepgflibrary{patterns}
\usetikzlibrary{arrows,shapes,patterns.meta,decorations.pathmorphing,decorations.text,positioning,arrows,decorations.markings,shapes.arrows} 
\pgfplotsset{compat=1.10}

\tikzdeclarepattern{
	name=mylines,
	parameters={
		\pgfkeysvalueof{/pgf/pattern keys/size},
		\pgfkeysvalueof{/pgf/pattern keys/angle},
		\pgfkeysvalueof{/pgf/pattern keys/line width},
	},
	bounding box={
		(0,-0.5*\pgfkeysvalueof{/pgf/pattern keys/line width}) and
		(\pgfkeysvalueof{/pgf/pattern keys/size},
		0.5*\pgfkeysvalueof{/pgf/pattern keys/line width})},
	tile size={(\pgfkeysvalueof{/pgf/pattern keys/size},
		\pgfkeysvalueof{/pgf/pattern keys/size})},
	tile transformation={rotate=\pgfkeysvalueof{/pgf/pattern keys/angle}},
	defaults={
		size/.initial=5pt,
		angle/.initial=45,
		line width/.initial=.4pt,
	},
	code={
		\draw [line width=\pgfkeysvalueof{/pgf/pattern keys/line width}]
		(0,0) -- (\pgfkeysvalueof{/pgf/pattern keys/size},0);
	},
}

\newcommand{\xdx}[1]{{{\rm d}#1}}
\def\xdrv#1#2{\frac{{\rm d}#1}{{\rm d}#2}}
\def\cf{\emph{cf.\/}}\def\eg{\emph{e.g.\/}}

\def\@Rref#1{\hbox{\rm \ref{#1}}}
\def\Rref#1{\@Rref{#1}}
\def\xC{{\rm C}}

\def\xCinfty{{\rm C}^{\infty}} 

\def\xH{{\rm H}}

\def\xWn#1{{\rm W}^#1}

\def\xLtwo{{\rm L}^{2}}

\def\xLinfty{{\rm L}^{\infty}} 

\def\xA{{\rm A}}

\def\xX{{\rm X}}
\def\xY{{\rm Y}}

\def\xE{{\rm E}}
\def\xF{{\rm F}}
\def\xB{{\rm B}}

\def\xO{{\rm O}}

\def\xM{{\rm M}}

\def\xG{{\rm G}}

\def\xK{{\rm K}}
\def\RR{{\mathbb{R}}}

\newcommand{\ssubset}{\subset\joinrel\subset}
\DeclareMathOperator*{\esssup}{ess\,sup}

\makeatletter
\newcommand*{\toccontents}{\@starttoc{toc}}
\makeatother

\numberwithin{equation}{section}
\theoremstyle{plain}
\newtheorem{thrm}{Theorem}[section]
\newtheorem*{thrmB}{Theorem}
\newtheorem{lmm}[thrm]{Lemma}

\newtheorem{prpstn}[thrm]{Proposition}

\newtheorem{ssmptn}[thrm]{Assumption}

\theoremstyle{definition}
\newtheorem{dfntn}[thrm]{Definition}
\newtheorem{cndtn}{Condition}
\newtheorem{xmpl}[thrm]{Example}
\newtheorem{rmrk}[thrm]{Remark}

\theoremstyle{plain}

\begin{document}\linespread{1.06}\selectfont
	
	\author{Laurent Mertz\,\footnote{Department of Mathematics, City University of Hong Kong, Kowloon, Hong Kong, China, e-mail: \href{mailto:lmertz@cityu.edu.hk}{lmertz@cityu.edu.hk}}\and
		Vahagn~Nersesyan\,\footnote{NYU-ECNU Institute of Mathematical Sciences at NYU Shanghai, 3663 Zhongshan Road North, Shanghai, 200062, China, e-mail: \href{mailto:vahagn.nersesyan@nyu.edu}{Vahagn.Nersesyan@nyu.edu}}\and
		Manuel~Rissel\,\footnote{NYU-ECNU Institute of Mathematical Sciences at NYU Shanghai, 3663 Zhongshan Road North, Shanghai, 200062, China, e-mail: \href{mailto:Manuel.Rissel@nyu.edu}{Manuel.Rissel@nyu.edu}}
	}
	
	\date{}
	\title{Exponential mixing of constrained random dynamical systems via controllability conditions}
	\maketitle
	
	\begin{abstract}

		We provide deterministic controllability conditions that imply exponential mixing properties for randomly forced constrained dynamical systems with possibly unbounded state space. As an application, new ergodicity results are obtained for non-smooth models in elasto-plasticity driven by various types of noise, including white noise. It is thereby illustrated how tools from control theory can be utilized to tackle regularity issues that commonly arise in the qualitative study of constrained systems.

		\quad
		
		\begin{center}
		\end{center}
	
	\end{abstract}

	\setcounter{tocdepth}{1}
	\toccontents      	
	\newgeometry{margin=1.5545in}

	\section{Introduction}
 	The objective of this work is to show that deterministic controllability conditions can be used to infer ergodic properties of stochastic non-smooth constrained dynamics governed by differential inclusions. Such random dynamical systems (RDS), which are ubiquitous in science and engineering, are often described by variational inequalities and their study is obstructed by irregular coefficients (\cf~Chapter 4 in \cite{PardouxRascanu-2014}). Over many decades, several notions of controllability have been developed to measure the ability of dynamical systems to transition between prescribed states under certain degrees~of freedom (the controls). Here, classical controllability properties of deterministic systems will be exploited for the investigation of constrained stochastic systems that arise when substituting the controls by noise. On an abstract level, we thereby demonstrate that controllability techniques can be used to tackle regularity issues commonly associated with the study of constrained problems. To illustrate the interest of our work from a practical perspective, we consider systems under hysteresis (finite dimensional elasto-plasticity) with random forcing.
 
	\subsection{Review on stochastic elasto-plastic models}

   	Materials often react to forces exerted upon them through elastic (reversible) or plastic (permanent) deformations. As we commonly observe, bending and then straightening a metallic object, like a wire or spoon, induces a permanent deformation, a plastic effect, at the initial bend. For many materials in nature, alternating stresses or deformations typically affect the material's local properties; \eg, the maximum stress (yield strength) tolerated before plastic deformation results in state (integrity and strength) degradation. These are examples of the "Bauschinger effect", which can complicate the risk failure analysis of mechanical structures under vibrational forces. This has significant relevance in earthquake engineering, made even more challenging due to the random nature of seismic forces.

    Much of the work on stochastic elasto-plasticity was done by engineers \cite{RobertsSpanos-1990}, with first strides made in the 1960s \cite{KarnoppScharton-1966}, but using mainly heuristical approaches. Mathematically, the field is still in its infancy, thereby providing rich questions and motivations for the development of a rigorous theory. Stochastic variational inequalities (SVIs) have been identified as a solid mathematical framework \cite{BensoussanTuri-2006} for describing the dynamics of various elasto-plastic systems such as white noise driven elasto-plastic oscillators. The existence and uniqueness of invariant probability measures for SVIs modelling white noise driven elasto-plastic oscillators has been shown in \cite{BensoussanTuri-2008}. The proof consists of extending Khasminski\u{\i}'s method \cite{Khasminskii-1980}, which leads to the study of degenerate elliptic problems with non-local Dirichlet boundary conditions. 
           
    The existence of a unique invariant measure is useful in engineering, \eg, when evaluating statistics of plastic deformations at large times. Another application of interest is to study the frequency of occurrence of plastic deformations \cite{MertzFeau-2012}; see also \cite{BensoussanMertzYam-2012,BensoussanMertz-2012}. Large time statistics of plastic deformations contain crucial information for risk analysis of failure. Since closed formul\ae~are not available, a numerical approximation of the invariant measure by a deterministic algorithm has been proposed in \cite{BensoussanMertzPironneauTuri-2009}. The latter is based on a class of partial differential equations defining the invariant measure via duality, as introduced in \cite{BensoussanTuri-2010}. 

    \subsection{Review on controllability methods}\label{subsection:overview}

	In the literature, the ergodicity of stochastic differential equations has been predominantly studied when the vector fields entering the equation are smooth and the driving noise is white. Typically, it is assumed that the system possesses a Lyapunov function, the coefficients are sufficiently smooth, and the H\"ormander condition is satisfied everywhere in the phase space. Then the transition function of the solution process has a smooth and almost surely positive density relative to the Lebesgue measure. As a result, the process is strong Feller and irreducible, leading, by virtue of Doob's theorem, to the existence of a unique stationary measure (see \cite{MT-93, Khasminskii-1980,DaPratoZabczyk-1996}).  
		
	Controllability-type arguments allow to considerably relax the assumptions on both the deterministic and stochastic counterparts of the system. Arnold and Kliemann~\cite{AK-87} are one of the firsts to explicitly use controllability terminology and methods to establish the uniqueness of stationary measures for degenerate diffusion equations. They assume that H\"ormander's condition is satisfied at one point and the process is irreducible; their proof extensively employs the Gaussian structure of the noise.  In a more refined application of control theory, Shirikyan~\cite{Sh-17} uses a coupling method to prove exponential mixing in the total variation metric for RDS on compact Riemannian manifolds. His approach relies on the solid controllability property from a point and the global approximate controllability to that point; the noise is assumed to satisfy a decomposability condition and is not restricted to being Gaussian. Extensions of this result to non-compact phase spaces were provided by Raqu\'epas~\cite{Raq-19} in the case of degenerate white noise and by Raqu\'epas and Nersesyan~\cite{NR-21} in the case of degenerate Poisson noise. The employed controllability approaches are quite flexible and admit infinite-dimensional generalizations as well; for instance, Kuksin et al.~\cite{KNS} develop a controllability method for studying the ergodicity of the Navier--Stokes system driven by a bounded noise.
		  	  
	In the current paper, motivated by applications to elasto-plastic models driven by random forces, we further develop the abstract criteria for exponential mixing provided in~\cite{Sh-17}. The main novelty of the version presented here is that the underlying deterministic system is allowed to be non-smooth. Specifically, the differential inclusion modeling elasto-plasticity leads to a lack of time-differentiability at certain instances of time, and the drift of the system is assumed to be differentiable only near an interior point~$p$. 
	To address these difficulties, inspired by the techniques of~\cite{Sh-17},  we assume that the system is solidly controllable from the point~$p$, and the dynamics is smooth in a neighborhood of $p$ which is accessible from everywhere in the phase space. We verify these assumptions for the example of an elasto-plastic system driven by white noise or general decomposable noise.

 	\subsection{Overview of the manuscript}
	
 	Let us denote $\xM \coloneq \mathbb{R} \times [-1,1]$ and assume that $f\colon \xM \longrightarrow \mathbb{R}$ is a locally Lipschitz function. We consider a class of non-smooth dynamical systems of the form
 	
	\begin{equation}\label{equation:Application1_Model0.1}
			\dot{y} = f(y,z) + \zeta, \quad
			y \in \dot{z} + \partial g(z),
	\end{equation}
	where $g\colon \mathbb{R}\longrightarrow \mathbb{R} \cup \{+\infty\}$ is the characteristic function (in the sense of convex analysis) of the interval~$[-1,1]$ and~$\partial g$ is its subdifferential. Our present work covers a broad class of noise $\zeta(\cdot)$; but to fix the ideas, let us assume in this introduction that~$\zeta$ is a white noise of the form $\zeta(t) = \dot{\beta} (t)$, where $\beta(t)$ is a standard Brownian motion. Furthermore, the system \eqref{equation:Application1_Model0.1} is supplemented with the initial condition
	\begin{equation}\label{equation:Application1_Model_initialcondition0.1}
		(y, z)(0) = (y_0,z_0) \in  \xM.
	\end{equation}
	To have globally well-defined and dissipative dynamics, we assume that the following Lyapunov-type condition is satisfied for the drift
		  	   
	\begin{equation}\label{equation:Lyapunov}
					y f(y,z) \leq - \alpha y^2 + C
	\end{equation}
	with some constants~$\alpha,C > 0$ and any $(y,z) \in \xM$. Under the above conditions, we establish the following result.
	
	\begin{thrmB}If~$f$ is smooth (infinitely differentiable) in a neighborhood of some interior point $p \in \mathbb{R} \times (-1,1)$,
	then the Markov process associated with the problem \eqref{equation:Application1_Model0.1}, \eqref{equation:Application1_Model_initialcondition0.1} has a unique stationary measure that is exponentially mixing.
	\end{thrmB}
	
	To outline the idea of the proof of this theorem, let us first note that the problem can be reduced to the study of a discrete-time RDS in the following way. Let us fix a time $T_0 > 0$ and let~$\xE$ be the Banach space $\xC_0([0,T_0];\mathbb{R})$ of continuous functions $\eta \colon [0,T_0] \longrightarrow \mathbb{R}$ with $\eta(0) = 0$, endowed with the uniform norm. Furthermore, let us denote by $S$ the resolving mapping of our problem:
	\[
		S\colon \xM\times \xE\longrightarrow \xM, \quad  (x_0, \eta)\mapsto x(T_0),   
	\]
	where $x(t)\coloneq (y(t),z(t))$ solves \eqref{equation:Application1_Model0.1}, \eqref{equation:Application1_Model_initialcondition0.1} with the initial state $x_0 =(y_0,z_0)$ and force $\zeta \coloneq ``\dot{\eta}"$. We define a family of independent identically distributed~(i.i.d) random variables $\{\eta_k\}_{ k \geq 1}$ in $\xE$ through
	\[
		\eta_k(t)\coloneq  \beta((k-1)T_0 +t)-\beta((k-1)T_0), \quad t\in [0,T_0]
	\] 
	and note that the discrete-time process $x_k\coloneq x(kT_0)$, obtained by restricting our original process at integer times, satisfies the relations
	\begin{equation}\label{equation:System0.1}
		x_k = S(x_{k-1};\eta_k).
	\end{equation}
	In this way, the problem of mixing for \eqref{equation:Application1_Model0.1}, \eqref{equation:Application1_Model_initialcondition0.1} is transformed to a similar problem for the discrete-time RDS \eqref{equation:System0.1}.
	The exponential mixing for the latter is derived from the following four properties by appropriately extending the abstract theory developed in~\cite{Sh-17}:
	\begin{enumerate}
		\item[$\bullet$] Lyapunov structure;
		\item[$\bullet$] approximate controllability to an interior point $p$;
		\item[$\bullet$] smoothness of $S$ near $p$ and solid controllability from $p$,
		\item[$\bullet$] decomposability of the law of $\eta_k$.
	\end{enumerate}
	The Lyapunov property follows from the assumption \eqref{equation:Lyapunov}. Approximate controllability is established essentially by using explicit formulas for the control. The regularity and solid controllability are obtained by choosing $T_0$ sufficiently small and observing that the problem is exactly controllable with a control function that depends continuously on the target. Decomposability of the law of the Brownian motion is shown to hold with respect to an increasing family of finite-dimensional subspaces formed by trigonometric functions. Finally, the exponential mixing at discrete times, combined with the above Lyapunov-type assumption, implies the exponential mixing for the continuous-time process.

 	\subsection*{Organization}
 	The goal of \Cref{section:main} is to develop general conditions that are sufficient for exponential mixing; the abstract main result of this article is stated in \Cref{subsection:mainresult} and proved in \Cref{subsection:proof_mainresult}. \Cref{section:elastoplasticity} showcases~concrete applications in finite-dimensional elasto-plasticity. Finally, exponential recurrence estimates and a measure transformation theorem are collected in Appendices~\Rref{section:appendix_exponentialrecurrence} and~\Rref{section:appendix_measuretransformation}.
	
	\subsection*{Notations}
	Let $(\xX, d)$ be a complete separable metric space and $\xB_{\xX}(a, R)$ the open ball in $(\xX, d)$ of radius~$R > 0$ centered at~$a \in \xX$. Moreover, the $\xX$-closure of $\xA \subset \xX$ is indicated by $\operatorname{clos}_{\xX} \xA$, the Borel~$\sigma$-algebra on $(\xX, d)$ is written as~$\mathcal{B}(\xX)$, and the Banach space $\xLinfty(\xX;\RR)$ of essentially bounded $\mathcal{B}(\xX)$-measurable functions $f\colon \xX\longrightarrow \mathbb{R}$ is endowed with the uniform norm
	\[
		\|f\|_\infty \coloneq \esssup_{x\in \xX} |f(x)|, \quad f \in \xLinfty(\xX;\RR).
	\]
	Further, the Borel probability measures $\mathcal{P}(\xX)$ on $\xX$ are equipped with the total variation distance
	\begin{equation}\label{equation:var}
		\begin{aligned}
			\|\mu_1-\mu_2\|_{\operatorname{var}} & \coloneq \sup_{\Gamma\in \mathcal{B}(\xX)} | \mu_1(\Gamma)- \mu_2(\Gamma)|\\
			& \, =
			\frac{1}{2}\sup_{\|f\|_\infty\leq 1}|\langle f,\mu_1\rangle_X-\langle f,\mu_2\rangle_{\xX}|, \quad \mu_1, \mu_2\in \mathcal{P}(\xX),
		\end{aligned}
	\end{equation}
	where the bracket
	\[
		\langle f, \mu \rangle_{\xX} \coloneq \int_{\xX} f(x)\,\mu(\xdx{x})
	\]
	is well-defined, in particular, for all $f\in \xLinfty(\xX;\RR)$ and $\mu \in \mathcal{P}(\xX)$. Regarding the basic property expressed by the second equation of \eqref{equation:var}, see, for instance, \cite[Exercise 1.2.10]{KS-12}. Furthermore, when $\mu_1$ and $\mu_2$ are absolutely continuous with respect to a measure $\mu \in \mathcal{P}(\xX)$, with densities $\rho_1$ and $\rho_2$, one has (\cf~\cite[Proposition 1.2.7]{KS-12})
	\begin{equation}\label{equation:varnormdensityidentiy}
		\|\mu_1-\mu_2\|_{\operatorname{var}} = 1 - \int_{\xX} \min\{\rho_1(x), \rho_2(x)\} \, \mu(\xdx{x}).
	\end{equation}
    Throughout, the symbol $C$ refers to unessential positive constants that may change during the estimates, and we~abbreviate
	\[
		\mathbb{R}_+ \coloneq [0,\infty), \quad \mathbb{R}_* \coloneq \mathbb{R}\setminus\{0\}, \quad \mathbb{N}_0 \coloneq \mathbb{N} \cup \{0\}.
	\]

	\section{Exponential mixing via controllability}\label{section:main} 
	This section presents sufficient conditions for the exponential mixing of a discrete-time~RDS on a connected smooth Riemannian manifold~$\xM$ with or without boundary.
	The plan is as follows.
	\begin{itemize}
		\item[$\triangleright$] \Cref{subsection:fourconditions}. Four conditions for exponential mixing are introduced.
		\item[$\triangleright$] \Cref{subsection:mainresult}. The abstract main result (\Cref{theorem:main}) is presented.
		\item[$\triangleright$]
		 \Cref{subsection:proof_mainresult}. \Cref{theorem:main} is proved.
	\end{itemize}

	\subsection{Abstract framework}\label{subsection:fourconditions}
	Let~$d_{\xM}\colon \xM\times \xM \longrightarrow \mathbb{R}_+$ be the natural Riemannian distance and suppose that~$(\xM,d_{\xM})$ forms a complete metric space. Given a separable Banach space $\xE$ (\enquote{noise space}), a continuous mapping
	\[
		S \colon \xM \times \xE \longrightarrow \xM,
	\]
	and a sequence $\{\eta_k\}_{k\in\mathbb{N}}$ of i.i.d. random variables in~$\xE$, we consider the RDS defined via
	\begin{equation}\label{equation:generalRDS}
		x_k = S(x_{k-1}; \eta_k), \quad k \geq  1.
	\end{equation}
	This RDS gives rise to a Markov family $\{x_k, \mathbb{P}_x\}_{k \in \mathbb{N}_0}$ and the associated transitions~$\{P_k(x,\Gamma)\}_{k \in \mathbb{N}_0}$ parametrized by the initial condition~$x_0 = x \in \xM$; the corresponding Markov semigroups are for $k \geq 0$ given by
	\begin{align*}
		\mathfrak{P}_k\colon \xLinfty(\xM) & \longrightarrow \xLinfty(\xM), \quad \mathfrak{P}_k f (x) \coloneq \int_{\xM} f(y) P_k(x, \xdx{y}), \\
		\mathfrak{P}_k^*\colon\mathcal{P}(\xM) & \longrightarrow \mathcal{P}(\xM), \quad \mathfrak{P}_k^* \mu(\Gamma) \coloneq \int_{\xM} P_k(x,\Gamma) \mu(\xdx{x}).
	\end{align*}
	\begin{dfntn}
		A measure $\mu \in \mathcal{P}(\xM)$ is called stationary for the family $\{x_k, \mathbb{P}_x\}_{k \in \mathbb{N}_0}$ provided that~$\mathfrak{P}_1^*\mu = \mu$.
	\end{dfntn}

	\begin{dfntn}\label{definition:Lyapunov}
		A Lyapunov function for the RDS \eqref{equation:generalRDS} is understood as any continuous functional~$V\colon \xM\longrightarrow [1,+\infty)$ having compact level sets (that is,~$\{V(x) \leq R\} \ssubset \xM$ for all~$R$) such that there are~$q \in (0,1)$ and $A > 0$ with
		\begin{equation}\label{equation:abstractLyapunov}
			\mathbb{E}_x  V(x_1) \leq q V(x) + A
		\end{equation}
		for all $x \in \xM$. Here, the expectation with respect to $\mathbb{P}_x$ is denoted by $\mathbb{E}_x$; the state $x_1$ is obtained from the initial data $x \in \xM$ via \eqref{equation:generalRDS}.
	\end{dfntn}
	
	\begin{dfntn}\label{definition:exponentialmixing}
		A stationary measure $\mu \in \mathcal{P}(\xM)$ for the family $\{x_k, \mathbb{P}_x\}_{k \in \mathbb{N}_0}$ is exponentially mixing if there are constants $\gamma > 0$ and $C>0$, and a Lyapunov function~$V$ satisfying
		\begin{equation}\label{equation:definition_exponentialmixing}
			\|\mathfrak{P}_k^* \lambda -\mu\|_{\operatorname{var}} \leq C\operatorname{e}^{-\gamma k}  \langle V, \lambda \rangle_{\xM}, \quad k  \geq 1
		\end{equation}
	for any $\lambda\in \mathcal{P}(\xM)$ with $\langle V, \lambda \rangle_{\xM}< +\infty$.
	\end{dfntn}
	
	\subsubsection{Conditions}\label{subsubsection:conditions}
	We describe below four conditions on the mapping~$S$ and the sequence~$\{\eta_k\}_{k\in\mathbb{N}}$ introduced above; as stated in \Cref{theorem:main}, they shall guarantee the existence and exponential mixing of a stationary measure for the Markov family $\{x_k, \mathbb{P}_x\}_{k \in \mathbb{N}_0}$.
	
	\begin{cndtn}[Lyapunov structure]\label{condition:C1}
		There exists a Lyapunov function for \eqref{equation:generalRDS} in the sense of \Cref{definition:Lyapunov}.
	\end{cndtn}
	
	\begin{rmrk}
		As a consequence of \Cref{condition:C1}, the standard Bogolyubov--Krylov argument yields the existence of at least one stationary measure $\mu \in \mathcal{P}(\xM)$ for the family $\{x_k, \mathbb{P}_x\}_{k \in \mathbb{N}_0}$, and~$\langle V,\mu \rangle_{\xM} < +\infty$ follows via Fatou's lemma; \cf~Section~2.5 in~\cite{KS-12}.
	\end{rmrk}

	The uniqueness of a stationary measure essentially follows from controllability properties of the deterministic system described by $S$ when viewing~$\{\eta_k\}_{k \geq 1}$ in \eqref{equation:generalRDS} as controls. As mentioned in the introduction, we deviate from the abstract theory in \cite{Sh-17} by allowing also manifolds with boundary (the constraint) and non-smoothness of the map~$S$; to this end, we utilize controllability to an interior point $p$ near which~the system is smooth. In what follows, we write for simplicity~$S_2(x; \zeta_1, \zeta_2)$ instead of~$S(S(x;\zeta_1); \zeta_2)$ and analogously define $S_n(x; \zeta_1,\ldots,  \zeta_n)$ for any $n \geq 1$.

	\begin{cndtn}[Approximate controllability]\label{condition:C2}
	There is an interior point $p \in \xM \setminus \partial \xM$ such that for any
		\begin{itemize}
			\item accuracy parameter $\delta > 0$,
			\item radius $R>0$,
			\item initial state $x \in \operatorname{clos}_{\xM} \xB_{\xM}(p, R)$,
		\end{itemize} 
		there exist a \enquote{time} $n = n(R, \delta) \in \mathbb{N}$ and controls $\{\zeta_i\}_{i\in\{1,\dots,n\}} \subset \xE$ satisfying
		\begin{equation*}
				S_n(x;\zeta_1, \dots,\zeta_n)\in\xB_{\xM}(p, \delta). 
		\end{equation*}
	\end{cndtn}

	\begin{rmrk}
		We will utilize Conditions~\Rref{condition:C1} and \Rref{condition:C2} to show the exponential recurrence property detailed in \Cref{section:appendix_exponentialrecurrence}. 
	\end{rmrk}

	We further resort to the notion of solid controllability introduced by Agrachev and Sarychev in \cite{AS-05} for the regulation of incompressible fluids by means of low mode forcing. Solid controllability is a type of (local) exact controllability that is robust under small perturbations. It has already proved to be useful in a regular setting when studying ergodicity properties of RDS associated with various models; e.g., see \cite{AKSS-07, Sh-07, Sh-17, Raq-19, NR-21}.

	\begin{cndtn}[Solid controllability]\label{condition:C3}
		Let $p \in \xM \setminus \partial \xM$ be as in \Cref{condition:C2}. There are   
		numbers $\epsilon, \delta > 0$, non-empty open balls
		$\widetilde{\xB}_{\xM}, \widehat{\xB}_{\xM} \subset \xM \setminus \partial \xM$ and $\widetilde{\xB}_{\xE} \subset \xE$, and a compact set $\mathcal{K}_{\xE} \subset \widetilde{\xB}_{\xE}$ such that the following properties hold.
		
		\begin{enumerate}[(i)]
			\item {\it (Regularity).} Given any $x\in \operatorname{clos}_{\xM}\xB_{\xM}(p, \delta)$, one has $S(x;\widetilde{\xB}_{\xE}) \subset \widetilde{\xB}_{\xM}$, the mapping $S(x;\cdot)\colon\widetilde{\xB}_{\xE} \longrightarrow\widetilde{\xB}_{\xM}$ is Fr\'echet differentiable with continuous derivative~$(x,\zeta) \mapsto D_{\zeta}S(x,\zeta)$ on~$\xB_{\xM}(p, \delta)\times \widetilde{\xB}_{\xE}$, and~$S(p;\cdot)\colon \widetilde{\xB}_{\xE} \longrightarrow \widetilde{\xB}_{\xM}$ is infinitely Fr\'echet differentiable.
			
			\item {\it (Solid controllability).} For any continuous function $\Phi\colon \mathcal{K}_{\xE} \longrightarrow \xM$, one has the implication
			\[
				\sup_{\zeta \in \mathcal{K}_{\xE}} d_{\xM}(\Phi(\zeta),S(p;\zeta)) \leq \epsilon \quad \Longrightarrow \quad \widehat{\xB}_{\xM} \subset \Phi(\mathcal{K}_{\xE}).
			\]
		\end{enumerate}
	\end{cndtn}
	
	The next condition states that the noise in \eqref{equation:generalRDS} is decomposable. To make this more precise, let~$\ell$ be the law of~$\eta_1$ from \eqref{equation:generalRDS}. Furthermore, given any complemented subspace $\xF \subset \xE$, the projection onto $\xF$ in $\xE$ is denoted by~$\mathsf{P}_{\xF}$ and the image of $\ell$ under $\mathsf{P}_{\xF}$ is written as $(\mathsf{P}_{\xF})_* \ell \coloneq \ell \circ \mathsf{P}_{\xF}^{-1}$.

	\begin{cndtn}[Decomposability]\label{condition:C4}
		Let $\mathcal{K}_{\xE}\subset \xE $ be as in \Cref{condition:C3}.
		There exists an increasing sequence of finite-dimensional subspaces $\{\xF_j\}_{j \in \mathbb{N}} \subset \xE$ with respective complements $\{\xH_j\}_{j \in \mathbb{N}}$ in~$\xE$ such that
		\begin{itemize}
			\item[(a)] $\cup_{j=1}^{+\infty} \xF_j$ is dense in $\xE$,
			\item[(b)] $\ell = [(\mathsf{P}_{\xF_j})_*\ell] \otimes [(\mathsf{P}_{\xH_j})_*\ell]$ for any $j \in \mathbb{N}$,
			\item[(c)] $(\mathsf{P}_{\xF_j})_*\ell$ possesses for each $j \in \mathbb{N}$ a positive continuous density with respect to the Lebesgue measure on $\xF_j$,
			\item[(d)] $\lim_{j\to+\infty}\mathsf{P}_{\xF_j}\zeta = \zeta$ uniformly with respect to $\zeta\in \mathcal{K}_{\xE}$, that is
			\[
				\sup \limits_{\zeta \in \mathcal{K}_{\xE}} \| \mathsf{P}_{\xF_j} \zeta - \zeta \|_{\xE} \to 0, \: \mbox{ as } \: j \to \infty.
			\]
 		\end{itemize}	
	\end{cndtn}

	\begin{rmrk}
		Conditions~\Rref{condition:C3} and \Rref{condition:C4} enable the application of a measure transformation theorem from~\cite{Sh-07} that provides a uniform estimate for the transition function~$P_1(x,\cdot)$ when $x$ is close to the point~$p$ (\cf~the proof of \Cref{lemma:transition_bound}). 
	\end{rmrk}

	\begin{xmpl}\label{Example:ONBRepCondition4}
		In addition to the case of white noise introduced in \Cref{subsection:overview}, we can treat general decomposable noise. Namely, when $\xE$ is a separable Hilbert space, \Cref{condition:C4} is verified for the law~$\ell$ of any random variable that admits a representation
		\begin{equation}\label{equation:etarep}
			\eta = \sum_{j=1}^\infty b_j\xi_j e_j,
		\end{equation}
		where
		\begin{itemize}
			\item $\{b_j\}_{j\in\mathbb{N}} \subset \mathbb{R}_*$ are square-summable,
			\item  $\{\xi_{j}\}_{j\in\mathbb{N}}$ are independent scalar random variables having a common positive continuous density~$\rho$ with respect to the Lebesgue measure,
			\item $\{e_j\}_{j\in\mathbb{N}}$ form an orthonormal basis in~$\xE$.
		\end{itemize}
		
        In fact, assuming \eqref{equation:etarep}, even a stronger version of (d) in \Cref{condition:C4} holds; namely, the operator sequence~$\{\mathsf{P}_{\xF_j}\}_{j \in \mathbb{N}}$ converges to the identity in the norm topology, that is
        \[
        	\sup \limits_{\| \zeta \|_{\xE} \: \leq \: 1} \| \mathsf{P}_{\xF_j} \zeta - \zeta \|_{\xE} \to 0, \: \mbox{ as } \: j \to \infty.
        \]
	\end{xmpl}
        A concrete example of a RDS satisfying the four conditions above is given in \Cref{section:elastoplasticity}. There, we will also discuss several common types of noise that satisfy Condition 4.
        
	\subsection{Main result}\label{subsection:mainresult}
	
	The core result of this section allows to conclude exponential mixing, and thus ergodicity, from the general conditions listed in \Cref{subsubsection:conditions}; it can be stated in the following way.
	
	\begin{thrm}\label{theorem:main}
		Under the Conditions~\Rref{condition:C1}--\Rref{condition:C4}, the family $\{x_k, \mathbb{P}_x\}_{k \in \mathbb{N}_0}$ has a unique stationary measure $\mu \in \mathcal{P}(\xM)$ that is exponentially mixing in the sense of \Cref{definition:exponentialmixing}.
	\end{thrm}

	The main novelty of \Cref{theorem:main} is that the underlying deterministic system is allowed to be non-smooth in several ways. First, the differential inclusion modeling elasto-plasticity induces irregularities (lack of differentiability of the time-dependent solution at some particular times). Second, the drift of the system is assumed to be differentiable only near an interior point~$p \in \xM \setminus \partial \xM$. To accommodate these difficulties, we exploit controllability techniques. Namely, since the interior point~$p$ in~\Cref{condition:C3} -- from which the RDS is solidly controllable -- can be reached from everywhere via \Cref{condition:C2}, smoothness of $S$ will be required only near $p$. Moreover, we provide a unified framework for treating white and general decomposable noise. 

	\subsection{Proof of the main result}\label{subsection:proof_mainresult}
	
	Inspired by the approach developed in \cite{Sh-17}, the proof of \Cref{theorem:main} shall involve a combination of controllability and coupling arguments. The argument is structured as follows.
	\begin{itemize}
		\item (\Cref{subsection:tf}). An estimate for the transition function $P_1$ is proved.
		\item (\Cref{subsubsection:cs}). Coupling arguments yield an auxiliary process with improved pathwise properties that has the same law as the original one.
		\item (\Cref{subsubsection:completion_proof_mainresult}). The proof is concluded.
	\end{itemize}
	
	\subsubsection{Transition function near $p$}\label{subsection:tf}
	
	The solid controllability and decomposability properties stated in Conditions~\Rref{condition:C3} and~\Rref{condition:C4} yield the subsequent bound for the transition function $P_1$. 
	\begin{lmm}\label{lemma:transition_bound}
		There are numbers  $\widehat{\delta} > 0$  and $r \in (0,1)$ such that
		\begin{equation}\label{equation:transitionbound}
			\left\| P_1(x, \,\cdot\,)  -  P_1(x', \,\cdot\,) \right\|_{\operatorname{var}} \leq r
		\end{equation}
		for any  $x,x'\in \operatorname{clos}_{\xM}\xB_\xM(p,\widehat{\delta})$.
	\end{lmm}
	
	\begin{proof}
		To begin with, we fix the objects $\epsilon, \delta, \widetilde{\xB}_{\xE}$, and $\mathcal{K}_{\xE}$ as in \Cref{condition:C3}. In particular,
		after reducing $\delta$ if necessary, we can assume that $\operatorname{clos}_{\xM} \xB_\xM(p, \delta)$ does not intersect the boundary~$\partial \xM$ of $\xM$.  	  
		
		\paragraph{Step 1. Fixing a projection.}
		Owing to the assumption (d) in \Cref{condition:C4}, combined with the continuity of $S(p;\,\cdot\,) \colon \xE\longrightarrow \xM$ and the compactness of $\mathcal{K}_{\xE}$, we obtain the bound
		\[
			\sup_{\zeta \in \mathcal{K}_{\xE}} d_M(S(p;\mathsf{P}_{\xF_j} \zeta),S(p;\zeta)) \leq \epsilon
		\]
		for any sufficiently large $j  \geq 1$. Hereafter, we fix such a number $j \in \mathbb{N}$.
		\paragraph{Step 2. Utilizing solid controllability.} In order to employ the assumption (ii) of \Cref{condition:C3}, we define a continuous function $\mathcal{K}_{\xE}\longrightarrow \xM$ by means of
		\[
			\zeta \mapsto \Phi(\zeta) \coloneq S(p;\mathsf{P}_{\xF_j} \zeta).
		\]
		Then, one can select a ball~$\xO \subset \xF_j$ so that the image of~$S(p;\,\cdot\,)\colon \xO\longrightarrow \xM$ contains the ball $\widehat{\xB}_{\xM} \subset \xM \setminus \partial \xM$. By increasing $j$ if necessary, thanks to \Cref{condition:C4}-(d), we have $\mathsf{P}_{\xF_j}(\mathcal{K}_{\xE}) \subset \widetilde{B}_{\xE}$. Therefore, one may assume that $\xO \subset \widetilde{\xB}_{\xE}$; hence, the mapping~$S(p;\,\cdot\,)\colon \xO\longrightarrow \xM$ is smooth. Thanks to Sard's theorem, there exists an element~$\widehat \zeta\in \xO$ such that $(D_\zeta S)(p;\widehat \zeta)$ has full rank. After possibly reducing $\delta$, we can further suppose that~$\xB_\xE(\widehat \zeta, \delta)\subset \widetilde{\xB}_{\xE}$. 
		\paragraph{Step 3. Measure transformation.} Let $\operatorname{vol}_\xM(\,\cdot\,)$ denote the Riemannian measure on~$\xM$. Owing to the previous step, we can now apply 
		\Cref{theorem:measuretransformation} to the mapping
		\[
			F \coloneq S \colon \xX \times  \xE \longrightarrow \xM, \quad \xX \coloneq \operatorname{clos}_{\xM} \xB_\xM(p, \delta).
		\]
		As a result, there is a number~$\widehat{\delta} > 0$ and a continuous function
		\[
			\psi\colon \operatorname{clos}_{\xM} \xB_\xM(p,\widehat{\delta})\times \xM\longrightarrow\mathbb{R}_+
		\]
		with
		\[	
			\psi(p,\widehat y) > 0, \quad \left(S(x;\,\cdot\,)_*\ell\right)(\xdx{y})  \geq \psi(x, y) \operatorname{vol}_\xM(\xdx{y})
		\]
		for $\widehat{y} \coloneq S(p;\widehat{\zeta})$ and all $x \in    \operatorname{clos}_{\xM}\xB_\xM(p,\widehat{\delta})$. Even more, by possibly shrinking $\widehat{\delta}$, one can choose a small number~$\varepsilon > 0$ so that $\psi(x,y)  \geq \varepsilon > 0$ for each $x\in \operatorname{clos}_{\xM}\xB_\xM(p,\widehat{\delta})$ and $ y\in \operatorname{clos}_{\xM}\xB_\xM(\widehat{y},\widehat{\delta})$. Consequently, the inequality
		\[
			\left(S(x;\,\cdot\,)_*\ell\right)(\xdx{y}) \geq \varepsilon \, \mathbb{I}_{\xB_\xM(\widehat y,\widehat{\delta})}(y) \operatorname{vol}_\xM(\xdx{y})
		\]
		holds for all points $x\in \operatorname{clos}_{\xM}\xB_\xM(p,\widehat{\delta})$. Hence, by using the representation~\eqref{equation:varnormdensityidentiy}, while recalling that $S(x;\,\cdot\,)_*\ell=P_1(x,\,\cdot\,)$, we arrive at the estimate
		\[	
				\|P_1(x,\,\cdot\,)-P_1(x',\,\cdot\,)\|_{\mbox{var}} \leq 1-\varepsilon \operatorname{vol}_\xM(\xB_\xM(\widehat y,\widehat{\delta})),
		\]
		and set $r \coloneq 1-\varepsilon \, \operatorname{vol}_\xM(\xB_\xM(\widehat y, \widehat{\delta}))$.
	\end{proof}
	
	\subsubsection{Coupling construction}\label{subsubsection:cs}
	
	Given two states $x,x'\in \xM$, let $\{x_k\}_{k \in \mathbb{N}}$ and $\{x_k'\}_{k \in \mathbb{N}}$ be the trajectories of \eqref{equation:generalRDS} issued from~$x$ and~$x'$, respectively. As $\{(x_k,x_k')\}_{k \in \mathbb{N}}$ might not be contractive, we resort to a coupling method in order to constract instead an auxiliary process $\{(\widetilde{x}_k, \widetilde{x}_k')\}_{k \in \mathbb{N}}$, such that $\{\widetilde{x}_k\}_{k \in \mathbb{N}}$ and $\{\widetilde{x}_k'\}_{k \in \mathbb{N}}$ have the same laws as $\{x_k\}_{k \in \mathbb{N}}$ and $\{x_k'\}_{k \in \mathbb{N}}$ respectively, but possesses better pathwise properties. More specifically, the goal is to choose $\{(\widetilde{x}_k, \widetilde{x}_k')\}_{k \in \mathbb{N}}$ such that~$\widetilde{x}_k$ and~$\widetilde{x}_k'$ coincide after a random time~$\sigma$ of finite exponential moment; it will be seen in \Cref{subsubsection:completion_proof_mainresult} that the existence of such a $\sigma$ almost immediately implies exponential mixing. 
	
	\paragraph{The auxiliary process.} We will define the new process $\{(\widetilde{x}_k,\widetilde{x}_k')\}_{k \in \mathbb{N}}$ with the help of coupling operators $\mathcal{R}_i\colon \xM\times \xM\longrightarrow \xM$, $i\in\{1,2\}$.
	Their construction is separated into the following three cases.
	\begin{itemize}
		\item If $x=x'$, we set $\mathcal{R}_1(x,x')=\mathcal{R}_2(x,x')=S(x;\eta_1)$.
		\item If $x\neq x'$ and~$x,x'\in \operatorname{clos}_{\xM}\xB_\xM(p,\widehat{\delta})$, where $\widehat{\delta}>0$ is fixed via \Cref{lemma:transition_bound}, then $(\mathcal{R}_1(x,x'),\mathcal{R}_2(x,x'))$ is chosen (by the method described in \cite{KS-12}) as a maximal coupling for the pair $(P_1(x,\,\cdot\,),P_1(x',\,\cdot\,))$. More specifically, we apply~\cite[Theorem~1.2.28]{KS-12} with
		\begin{gather*}
			X \coloneq \xM, \quad Z \coloneq \operatorname{clos}_{\xM}\xB_\xM(p,\widehat{\delta})\times \operatorname{clos}_{\xM}\xB_\xM(p,\widehat{\delta}),\\
			(\mu_1(z,\,\cdot\,),\mu_2(z,\,\cdot\,))\coloneq (P_1(x,\,\cdot\,),P_1(x',\,\cdot\,)), \quad  z \coloneq (x,x')\in Z,
		\end{gather*}which provides two random variables $\mathcal{R}_1(x,x')$ and $\mathcal{R}_2(x,x')$,  with respective distributions $P_1(x,\,\cdot\,)$ and  $P_1(x',\,\cdot\,)$, such that 
		\[
			\mathbb{P}\left\{\mathcal{R}_1(x,x')\neq \mathcal{R}_2(x,x')\right\} =\| P_1(x,\,\cdot\,)-P_1(x',\,\cdot\,)\|_{\operatorname{var}}.
		\]
  		\item Otherwise, we set  $\mathcal{R}_1(x,x')=S(x;\eta)$ and $\mathcal{R}_2(x,x')=S(x';\eta')$, where~$\eta$ and~$\eta'$ are independent copies of the random variable $\eta_1$.
	\end{itemize}
	All above-mentioned random variables are without loss of generality defined on the same probability space $(\Omega, \mathcal{F}, \mathbb{P})$. Next, in order to construct~$\{(\widetilde x_k,\widetilde x_k')\}_{k \in \mathbb{N}}$, we take independent copies $\{(\Omega_k, \mathcal{F}_k, \mathbb{P}_k)\}_{k \in \mathbb{N}}$  of $(\Omega, \mathcal{F}, \mathbb{P})$ and denote by $(\widehat\Omega, \widehat{\mathcal{F}}, \widehat{\mathbb{P}})$ their direct product. Then, for any~$x,x'\in \xM$, $k \in \mathbb{N}$, and $\omega=(\omega_1, \omega_2,\ldots)\in \widehat\Omega$, the anticipated process is given by
	\begin{gather*}
		\widetilde{x}_0 = x, \quad 	\widetilde x_k(\omega) = \mathcal{R}_1(\widetilde x_{k-1}(\omega),\widetilde x_{k-1}'(\omega), \omega_k),\\
		\widetilde{x}_0'=x', \quad  \widetilde{x}_k'(\omega)=\mathcal{R}_2(\widetilde x_{k-1}(\omega),\widetilde x_{k-1}'(\omega), \omega_k).
	\end{gather*}
	From this construction, it follows that the processes $\{\widetilde{x}_k\}_{k \in \mathbb{N}}$ and $\{\widetilde x_k'\}_{k \in \mathbb{N}}$ have the same laws as $\{x_k\}_{k \in \mathbb{N}}$ and $\{x_k'\}_{k \in \mathbb{N}}$ respectively. 
	
	\paragraph{The random time.} Owing to the above constructions, we define the desired random time via
	\begin{equation*}
		\sigma \coloneq \min\{k  \geq 0  \, | \,   \widetilde{x}_n = \widetilde{x}_n' \mbox{ for all $n \geq k$} \}.
	\end{equation*}
	Throughout, the convention $\min\varnothing\coloneq +\infty$ is used.	
 	\begin{lmm}[Finite exponential moment]\label{lemma:random_time}
		There are numbers~$\gamma>0$ and~$C>0$ such that for all $x, x'\in \xM$ one has the estimate
		\[
			\mathbb{E}_{(x,x')}\operatorname{e}^{\gamma \sigma} \leq C(V(x)+V(x')).
		\]
	\end{lmm}

	\begin{proof} 
		Let us denote by $\delta_0 > 0$ a number smaller than $\delta$ in \Cref{condition:C3} and $\widehat{\delta}$ from \Cref{lemma:transition_bound}; in what follows, we again write $\delta$ instead of $\delta_0$ and might further reduce the value of $\delta$ without changing the symbol. To begin with, a sequence of stopping times is given via
		\[
			\widetilde{\tau}_{\delta}(0) \coloneq 0, \quad \widetilde{\tau}_{\delta} (k) \coloneq \min\{n>\widetilde{\tau}_{\delta} (k-1) \,\,|\,\, \widetilde x_n,\widetilde x_n'\in \xB_\xM(p, \delta)\},\quad k \geq 1. 	
		\]
		By \Cref{lemma:recurrence2}, there are~$\varkappa>0$ and~$C>0$ with
		\begin{equation}\label{equation:varkappaC_0est}
			\mathbb{E}_{(x,x')} \operatorname{e}^{\varkappa \widetilde{\tau}_{\delta}(1)} \leq  C(V(x)+V(x'))
		\end{equation}
		for all $x,x'\in \xM$. Since~$V$ is continuous and $\operatorname{clos}_{\xM}\xB_\xM(p, \delta)$ compact, it follows that
		\begin{equation}\label{equation:boundedsupV}
				\sup_{x\in \xB_\xM(p, \delta)}|V(x)| \eqcolon D < \infty.
		\end{equation}
		Combining \eqref{equation:boundedsupV} with the inequality in \eqref{equation:varkappaC_0est}, the strong Markov property, and the fact that
		\[
			x_{\widetilde{\tau}_{\delta} (k-1)},x_{\widetilde{\tau}_{\delta} (k-1)}'\in \xB_\xM(p, \delta),
		\]
		one obtains the estimate
		\begin{equation}\label{equation:xxpineq}
			\begin{aligned}
				\mathbb{E}_{(x,x')} \operatorname{e}^{\varkappa \widetilde{\tau}_{\delta} (k)} & = \mathbb{E}_{(x,x')} \operatorname{e}^{\varkappa \widetilde{\tau}_{\delta} (k-1)}\mathbb{E}_{X(k)}
				\operatorname{e}^{\varkappa \widetilde{\tau}_{\delta} (1)}\\
				& \leq C \mathbb{E}_{(x,x')} \left( \left(V(x_{\widetilde{\tau}_{\delta} (k-1)})+V(x_{\widetilde{\tau}_{\delta} (k-1)}')\right) \operatorname{e}^{\varkappa \widetilde{\tau}_{\delta} (k-1)}\right)\\&\leq 2CD\mathbb{E}_{(x,x')} \operatorname{e}^{\varkappa \widetilde{\tau}_{\delta} (k-1)},	
			\end{aligned}
		\end{equation}
		where
		\[
			X(k)\coloneq (x_{\widetilde{\tau}_{\delta} (k-1)},x_{\widetilde{\tau}_{\delta} (k-1)}').
		\]
		Iterating \eqref{equation:xxpineq}, while utilizing~\eqref{equation:varkappaC_0est}, we~get
		\begin{equation}\label{equation:kappaineq}
			\mathbb{E}_{(x,x')}  \operatorname{e}^{\varkappa \widetilde{\tau}_{\delta} (k)} \leq C_0^k  (V(x)+V(x'))
		\end{equation}
		for $x,x'\in \xM$ and $C_0\coloneq 2CD$. Given the number $r \in (0, 1)$ from \Cref{lemma:transition_bound}, let us show that
		\begin{equation}\label{equation:subgoalrkbd}
			\mathbb{P}_{(x,x')}\{\sigma>\widetilde{\tau}_{\delta} (k+1)\}\leq r^k
		\end{equation}
		is true for all $k  \geq  0$. Indeed, the construction of the process $\{\widetilde x_k, \widetilde x_k'\}_{k \in \mathbb{N}}$ implies together with \eqref{equation:transitionbound} that
		\begin{multline}\label{equation:tbi}
			\mathbb{P}_{(x,x')}\left\{\widetilde x_{\widetilde{\tau}_{\delta} (k)+1}\neq \widetilde x_{\widetilde{\tau}_{\delta} (k)+1}'\right\}  \\
			\begin{aligned}
				& = \mathbb{P}_{(x,x')}\left\{\widetilde x_{\widetilde{\tau}_{\delta} (k)+1}\neq
				\widetilde x_{\widetilde{\tau}_{\delta} (k)+1}'|\widetilde x_{\widetilde{\tau}_{\delta} (k)}\neq \widetilde x_{\widetilde{\tau}_{\delta} (k)}'\right\} \mathbb{P}\left\{\widetilde x_{\widetilde{\tau}_{\delta} (k)}\neq
				\widetilde x_{\widetilde{\tau}_{\delta} (k)}'\right\}\\
				& \leq r\, \mathbb{P}_{(x,x')}\left\{\widetilde x_{\widetilde{\tau}_{\delta} (k)}\neq
				\widetilde x_{\widetilde{\tau}_{\delta} (k)}'\right\}\\
				& \leq r\, \mathbb{P}_{(x,x')}\left\{\widetilde x_{\widetilde{\tau}_{\delta} (k-1)+1}\neq \widetilde x_{\widetilde{\tau}_{\delta} (k-1)+1}'\right\}
			\end{aligned}
		\end{multline}
		for all  $k \geq 1$. Iterating \eqref{equation:tbi} yields
		\[
			\mathbb{P}_{(x,x')}\left\{\widetilde x_{\widetilde{\tau}_{\delta} (k)+1}\neq \widetilde x_{\widetilde{\tau}_{\delta} (k)+1}'\right\}\leq r^k, \quad k \geq 0.
		\]
		After noticing that
		\[
			\mathbb{P}_{(x,x')}\left\{\sigma>\widetilde{\tau}_{\delta} (k+1)\right\}\le\mathbb{P}_{(x,x')}\left\{\widetilde x_{\widetilde{\tau}_{\delta} (k)+1}\neq \widetilde x_{\widetilde{\tau}_{\delta} (k)+1}'\right\}, \quad k \geq 0,
		\]
		we obtain \eqref{equation:subgoalrkbd}, and the Borel--Cantelli lemma subsequently provides
		\[
			\mathbb{P}_{(x,x')}\left\{\sigma<+\infty\right\} = 1.
		\]
		Now, recalling that $\varkappa$ is the number from \eqref{equation:kappaineq}, a large $R > 0$ and a small $\gamma > 0$ are selected such that
		\[
			C_0^{1/R} r^{1-1/R} < 1, \quad R \gamma < \varkappa.
		\]
		Then, by combining H\"older's
		inequality with \eqref{equation:kappaineq} and \eqref{equation:subgoalrkbd}, one arrives at
		\begin{align*}
			\mathbb{E}_{(x,x')} \operatorname{e}^{\gamma \sigma}&\le1+ \sum_{k=0}^\infty \mathbb{E}_{(x,x')} \left(\mathbb{I}_{\{\widetilde{\tau}_{\delta} (k)<\sigma\le
				\widetilde{\tau}_{\delta} (k+1)\}} \operatorname{e}^{\gamma  \sigma}\right)
			\\&\le1+\sum_{k=0}^\infty \mathbb{E}_{(x,x')}
			\left(\mathbb{I}_{\{\widetilde{\tau}_{\delta} (k)<\sigma\leq \widetilde{\tau}_{\delta} (k+1)\}} \operatorname{e}^{\gamma \widetilde{\tau}_{\delta} (k+1)}\right)
			\\&\le
			1+\sum_{k=0}^\infty  \mathbb{P}_{(x,x')}\left\{\sigma>\widetilde{\tau}_{\delta} (k)\right\}^{1-\frac{1}{R}} \left(\mathbb{E}_{(x,x')} \operatorname{e}^{R\gamma 
				\widetilde{\tau}_{\delta} (k+1)}\right)^{\frac{1}{R}}
			\\&\leq 1+C_0^\frac{1}{R}r^{\frac{1}{R}-1}(V(x)+V(x'))\sum_{k=0}^\infty
			\left(r^{1-\frac{1}{R}}C_0^\frac{1}{R}\right)^k
			\\&\leq C (V(x)+V(x'))\quad\text{for $x,x'\in \xM$}.
		\end{align*}
	\end{proof}
	
	\subsubsection{Completion of the proof}\label{subsubsection:completion_proof_mainresult}
	
	To establish the exponential mixing estimate \eqref{equation:definition_exponentialmixing}, it suffices to show for any $k \geq 0$ and~$x,x' \in \xM$ the estimate
	\begin{equation}\label{equation:subgoalcompletionprf}
		\|\mathfrak{P}_k^* {\delta}_x - \mathfrak{P}_k^* {\delta}_{x'} \|_{\operatorname{var}}
		\leq C \operatorname{e}^{-\gamma k} (V(x) +V(x')),
	\end{equation}
	where $\delta_x$ refers to the Dirac measure concentrated at $x$.~In order to verify \eqref{equation:subgoalcompletionprf}, we take any $f \in \xLinfty(\xM)$ with~$\|f\|_{\infty}\leq 1$.~As the processes $\{(\widetilde x_k,\widetilde x_k')\}_{k \in \mathbb{N}}$ and~$\{(x_k,x_k')\}_{k \in \mathbb{N}}$ possess the same law, it follows that
	\begin{equation*}
		\mathfrak{P}_k f (x) - \mathfrak{P}_k f(x') = \mathbb{E}_{(x,x')} \left(f(\widetilde x_k) - f(\widetilde x_k') \right).
	\end{equation*}The assumption that $\|f\|_{\infty}\leq 1$, the Chebyshev inequality, and \Cref{lemma:random_time} imply that
	\begin{align*}
		\left|\mathbb{E}_{(x,x')} \left(f(\widetilde x_k) - f(\widetilde x_k')  \right)\right|&\leq 	  \mathbb{E}_{(x,x')}  | f(\widetilde x_k) - f(\widetilde x_k')  | 
		\\&= 	  \mathbb{E}_{(x,x')} \left(\mathbb{I}_{\{\widetilde x_k \neq  \widetilde x_k' \}}| f(\widetilde x_k) - f(\widetilde x_k')  |\right)
		\\&\leq 2 \mathbb{P}_{(x,x')}\{\widetilde x_k \neq  \widetilde x_k' \}
		\\&\leq 2 \mathbb{P}_{(x,x')}\{ \sigma>k \}
		\\&\leq 2C  \operatorname{e}^{-\gamma k} \left(V(x) +V(x')\right).
	\end{align*}
	As a result, (\cf~\eqref{equation:var})
	\begin{align*}
		\|\mathfrak{P}_k^* \delta_x - \mathfrak{P}_k^* \delta_{x'} \|_{\operatorname{var}}
		& = \frac12\sup_{ \|f\|_{\infty}\leq 1}      |\mathfrak{P}_k f(x) - \mathfrak{P}_k f(x')|  \\
		& \leq C  \operatorname{e}^{-\gamma k} (V(x) +V(x')),
	\end{align*}
	which completes the proof of \Cref{theorem:main}.

	\section{Low dimensional elasto-plastic models under random forces}\label{section:elastoplasticity} 
	
 	This section illustrates \Cref{theorem:main}'s effectiveness;  Conditions~\Rref{condition:C1}--\Rref{condition:C4} from \Cref{subsection:fourconditions} are verified for a range of non-smooth elasto-plastic dynamics driven decomposable noise or by white noise. In a systematic manner, new results on the uniqueness of stationary measures and exponential mixing are concluded. 
	
	\subsection{The model}\label{subsection:TheElastoPlasticModel}
	Let $\xM \coloneq \mathbb{R} \times [-1,1]$,  $g\colon \mathbb{R}\longrightarrow \mathbb{R} \cup \{+\infty\}$ the characteristic function of the compact interval~$[-1,1]$, and suppose that the drift $f\colon \xM \longrightarrow \mathbb{R}$ and noise $\zeta$ are fixed; specific requirements are introduced below. We then consider a class of non-smooth dynamical systems having the form
	\begin{equation}\label{equation:Application1_Model}
			\dot{y} = f(y,z) + \zeta, \quad
			y \in \dot{z} + \partial g(z), \quad (y, z)(0)=(y_0,z_0), 
	\end{equation}
	where the notation $\partial g$ stands for the subdifferential of $g$ which is defined as $\partial g (\mathfrak{z}) \coloneq \left \{ \xi \in \mathbb{R} \, | \, 
    \forall \: y \in \mathbb{R}\colon  g(\mathfrak{z}) + \xi (y-\mathfrak{z}) \leq g(y)  \right \}$. 
    Here, that is,
	\[
		g(z) = \begin{cases}
			0 & \mbox{if } z \in [-1,1],\\
			+\infty & \mbox{otherwise,}
		\end{cases} 
		\quad \quad
		\partial g(z) = \begin{cases}
			0 & \mbox{if } z \in (-1,1),\\
			\mathbb{R}_{\pm} & \mbox{if } z = \pm1,\\
			\varnothing & \mbox{otherwise.}
		\end{cases}
	\] 

    \begin{ssmptn}\label{assumption:Application1_force}
		The mapping $f\colon \xM \longrightarrow \mathbb{R}$ is locally Lipschitz, it is smooth in a neighborhood $\widetilde{\xB}_{\xM} \subset \xM \setminus \partial \xM$ of some point $p = (y_p, z_p) \in \mathbb{R} \times (-1,1)$, and there are~$\alpha,C > 0$ with
		\[
			y f(y,z) \leq - \alpha y^2 + C
		\]
		for all $(y,z) \in \xM$.
	\end{ssmptn}

	If $y_p = 0$ in \Cref{assumption:Application1_force}, we fix $T_0 = 1$. Otherwise, for technical reasons (see the proof of \Cref{proposition:ressc3}), a reference time $T_0 = T_0(f(p)) \in (0,1]$ and any radius~$r_0 > 0$ are fixed so small that solutions $(y,z)$ to \eqref{equation:Application1_Model} with $\zeta = 0$ satisfy
	\begin{equation}\label{equation:characterizationT0}
		(y_0,z_0) \in \xB_{\xM}(p, r_0) \quad \Longrightarrow \quad \forall t \in [0, T_0]\colon (y(t), z(t)) \in \widetilde{\xB}_{\xM}.
	\end{equation}
	To see that this choice is possible, note that, for initial states near $p$, the system \eqref{equation:Application1_Model} stays for a short time away from the plastic phase $z = \pm1$; namely, \eqref{equation:Application1_Model} with initial state $p$ can for small times be treated as a regular ODE or integral equation.  
	
	\subsection{Decomposable noise} \label{subsection:dn}
	We begin with the consideration of decomposable noise, where the process $\zeta$ satisfies the following assumption (\cf~\Cref{Example:ONBRepCondition4}).

	\begin{ssmptn}\label{assumption:Application1_noise}
		The collection $\{\eta(t)\}_{t \in \mathbb{R}_+}$ forms a real-valued random process that admits a representation
		\[
			\zeta(t) = \sum_{k=1}^{\infty} \mathbb{I}_{[(k-1)T_0,kT_0)}(t)\eta_k(t-(k-1)T_0),
		\]
		where $\{\eta_k\}_{k \in \mathbb{N}}$ are i.i.d. random variables taking values in $\xE \coloneq \xLtwo((0,T_0);\mathbb{R})$, obeying~$\mathbb{E} \|\eta_1\|_{\xE}^2 < + \infty$, and whose law $\ell$  satisfies \Cref{condition:C4}.
	\end{ssmptn}
	Owing to Assumptions~\Rref{assumption:Application1_force} and \Rref{assumption:Application1_noise}, the initial value problem described by \eqref{equation:Application1_Model} is globally well-posed \cite{PardouxRascanu-2014}. Given any state $x_0 = (y_0,z_0) \in \xM$, the associated solution $x = (y, z)$ to \eqref{equation:Application1_Model} induces a proper $\xM$-valued random process $\{x(t)\}_{t\in\mathbb{R}_+}$ such that
	\begin{equation*}\label{equation:MarkovFamily_elastoplastic1}
		\{x_k\}_{k\in\mathbb{N}_0} \coloneq \{x(kT_0)\}_{k\in\mathbb{N}_0}
	\end{equation*}
	forms a Markov family $\{x_k,\mathbb{P}_x\}_{k \in \mathbb{N}_0}$ with the associated Markov operators~$\mathfrak{P}_k$ and~$\mathfrak{P}_k^*$.
	
	\begin{thrm}\label{theorem:main_application_elastoplastic}
		Under Assumptions~\Rref{assumption:Application1_force} and \Rref{assumption:Application1_noise}, the Markov family $\{x_k,\mathbb{P}_x\}_{k \in \mathbb{N}_0}$ admits a unique stationary measure $\mu \in \mathcal{P}(\xM)$ which is exponentially mixing in the sense of \Cref{definition:exponentialmixing}
		with the Lyapunov function
		\begin{equation}\label{equation:lfa}
			V(y,z) = 1+y^2, \quad (y,z) \in \xM.
		\end{equation}
	\end{thrm}
	\begin{proof}
		Let $S\colon \xM \times \xE\longrightarrow \xM$, $(x_0, \zeta)\mapsto x(T_0)$ be the resolving operator for the system~\eqref{equation:Application1_Model}.~Then, by definition,
		\[
				x(kT_0) = S(x((k-1)T_0);\eta_k), \quad k  \geq 1.
		\]
		To check the hypotheses of \Cref{theorem:main},
		we first observe that \Cref{condition:C1} is satisfied with $V\colon \xM\longrightarrow [1,+\infty)$ from~\eqref{equation:lfa}. Indeed, this function is continuous and has compact level sets. By multiplying~\eqref{equation:Application1_Model} with its solution $y$, involving \Cref{assumption:Application1_force}, and resorting to Young's inequality, one finds
		\begin{equation}\label{equation:YEest}
				\frac{1}{2} \xdrv{}{t} y^2 \leq -\alpha y^2+C+\zeta y \leq - \frac{\alpha}{2} y^2  + C \left(1 + \zeta^2 \right), \quad \zeta\in \xE, t\in [0, T_0].
		\end{equation}
		Applying Gr\"onwall's inequality in \eqref{equation:YEest}, taking the~expectation, and utilizing the assumption that $\mathbb{E} \|\eta_1\|_{\xE}^2 < +\infty$, one arrives at the inequality~\eqref{equation:abstractLyapunov} with~$q = \operatorname{e}^{-\alpha}$; in particular, this implies \Cref{condition:C1}. 
        The approximate controllability to a distinguished point~$p \in \xM$, namely~\Cref{condition:C2}, and the solid controllability in~\Cref{condition:C3} are inferred from the respective Propositions~\Rref{proposition:Application1_exactcontrollability} and~\Rref{proposition:ressc3} established below. Next,  \Cref{condition:C4} is enforced via \Cref{assumption:Application1_noise}. Thus, in view of \Cref{theorem:main}, we obtain the existence of a unique stationary measure possessing the exponential mixing property~\eqref{equation:definition_exponentialmixing}.
	\end{proof}

	\subsection{A white noise example}\label{subsection:whitenoiseexample} 
	
	In this section, exponential mixing is established for the elasto-plastic system~\eqref{equation:Application1_Model} driven by white noise. That is, we consider the problem \eqref{equation:Application1_Model} with $\zeta = \dot{\beta}$, where $\beta$ is a standard Brownian motion.

	Given any $x = x_0 = (y_0, z_0) \in \xM$, let $\{x_t,\mathbb{P}_x\}_{t \geq 0}$ be the continuous-time Markov family (\cf~\cite[Section~1.3.3]{KS-12}) associated with the RDS~\eqref{equation:Application1_Model} for $\zeta = \dot{\beta}$; the corresponding Markov operators are then denoted as~$\mathfrak{P}_t$ and~$\mathfrak{P}_t^*$. Analogously to the discrete case, a measure~$\mu\in\mathcal{P}(\xM)$ is called stationary for the continuous-time Markov family~$(x_t,\mathbb{P}_x)$ if $\mathfrak{P}_t^*\mu=\mu$ holds for any $t \geq 0$.
	
	\begin{thrm}\label{theorem:main_application_elastoplasticBrownianMotion} 
		Under \Cref{assumption:Application1_noise},
		the family $(x_t,\mathbb{P}_x)$ has a unique stationary  measure $\mu \in \mathcal{P}(\xM)$, and there are positive numbers $\gamma$ and~$C$ such~that
		\[
				\|\mathfrak{P}_t^* \lambda -\mu\|_{\operatorname{var}}\leq 	C\operatorname{e}^{-\gamma t} \langle V, \lambda \rangle_\xM, \quad t \geq 0 
		\] 
		for any $\lambda\in \mathcal{P}(\xM)$ with $\langle V, \lambda  \rangle_\xM < +\infty$, where $V$ denotes the Lyapunov function specified in \eqref{equation:lfa}.
	\end{thrm}

	\begin{proof}
		The idea is to first prove exponential mixing via \Cref{theorem:main} for the system restricted to integer times. Subsequently, this can be generalized to continuous-time by using the Lyapunov structure.
		
		\paragraph{Step 1. Discrete-time.}
		Let $T_0 > 0$ be fixed as explained above \Cref{equation:characterizationT0}, and take $\xE$ as the separable Banach space $\xC_0([0,T_0];\mathbb{R})$ of continuous functions $\eta\colon [0,T_0] \longrightarrow \mathbb{R}$ with $\eta(0) = 0$; further, denote by $S$ the mapping
		\[
				S\colon \xM\times \xE\longrightarrow \xM, \quad  (x_0, \eta)\mapsto x(T_0),   
		\]
		where $x(t)\coloneq (y(t),z(t))$ solves \eqref{equation:Application1_Model} with the initial state $x_0 =(y_0,z_0)$ and driving force $\zeta \coloneq ``\dot{\eta}"$. Then a family of i.i.d. random variables $\{\eta_k\}_{ k \geq 1}$ is defined in $\xE$ via
		\[
				\eta_k(t)\coloneq  \beta((k-1)T_0 +t)-\beta((k-1)T_0), \quad t\in [0,T_0].
		\]
		In particular, given any $k \geq 1$, one has for $x_k\coloneq x(kT_0)$ the relation
		\[
			x_k = S(x_{k-1};\eta_k).
		\]
		In order to prove exponential mixing for the discrete-time Markov family $(x_k, \mathbb{P}_x)$, we need to check Conditions~\Rref{condition:C1}--\Rref{condition:C4}.
		
		Accounting for \eqref{equation:Application1_Model}, \Cref{assumption:Application1_force}, and It$\hat{\text o}$'s formula, then taking the expectation, it follows that
		\begin{equation*}
			 \xdrv{}{t} \mathbb{E}_x  y^2 =\mathbb{E}_x \left (2f(y,z)y+1\right)\leq -2\alpha \mathbb{E}_x y^2+2C+1.
		\end{equation*}
 		An application of Gr\"onwall's inequality then yields
		\begin{equation}\label{equation:c1iebrmo}
			\mathbb{E}_x V(x(t))\leq q^t V(x)+A, \quad t\in [0,T_0]
		\end{equation}
		with
		\[
			q\coloneq \operatorname{e}^{-2\alpha}, \quad A\coloneq (2C+1)/2\alpha.
		\]
		Thus, one obtains \Cref{condition:C1} by taking~$t=T_0$ in~\eqref{equation:c1iebrmo}. To verify~\Cref{condition:C2}, we apply \eqref{proposition:Application1_exactcontrollability} and integrate the obtained control with respect to time. For checking~\Cref{condition:C3}, one can repeat the proof of~\Cref{proposition:ressc3} below, followed by integrating the so-obtained control with respect to time. Moreover, due to~\Cref{proposition:ressc3}, the compact set~$\mathcal{K}\subset \xE$ in~\Cref{condition:C3} can be fixed as any appropriate closed ball in~$\xWn{{1,2}}_0([0,T_0];\mathbb{R})$. 
		 		 
		Finally, \Cref{condition:C4} is verified as in \cite[Appendix A]{Raq-19}. Hereto, we denote by~$\{\phi_j\}_{j \in \mathbb{N}}$ the trigonometric basis in $L^2((0,T_0);\mathbb{R})$; then, we set
		\[
		 	 e_j(t) \coloneq \int_0^t \phi_j(s) \, \xdx{s}, \quad t\in [0,T_0], \quad j \geq 1, \quad \xF_j \coloneq \operatorname{span}\{e_n \,| \, n\leq j \}.
	 	\]
		Since the space $\xWn{{1,2}}_0([0,T_0];\mathbb{R})$ is dense in $\xE$ and admits $\{e_j\}_{j \in \mathbb{N}}$ as an orthonormal basis, one can infer the property (a) in \Cref{condition:C4}. Furthermore, there are independent scalar standard normal random variables $\{\xi_j\}$ with (\cf~\cite[Section~3.5]{Bog-98})
		\[
		 	\beta(t)=\sum_{n=1}^{\infty} \xi_n e_n(t), \quad t \in [0,T_0].
		\]
		The independence of the sums $\sum_{n=1}^j \xi_n e_n$ and $\sum_{n=j+1}^{\infty} \xi_n e_n$ implies the representation~(b) in \Cref{condition:C4}. Property (c) is obvious, and property (d) follows from~\cite[Lemma~A.1]{Raq-19}.

		As a result, we conclude the existence of a unique stationary measure~$\mu \in \mathcal{P}(\xM)$  associated with $\mathfrak{P}_{T_0}^*$ that obeys $\langle V,\mu \rangle_\xM<+\infty$ and satisfies for some fixed positive numbers~$\gamma$ and~$C$ the estimate
		\begin{equation}\label{equation:resstep1sm}
			\|\mathfrak{P}_{kT_0}^* \lambda -\mu\|_{\operatorname{var}}\leq C\operatorname{e}^{-\gamma k} \langle V, \lambda  \rangle_\xM
		\end{equation}
		with arbitrary $k \geq 0$ and any $\lambda\in \mathcal{P}(\xM)$ that obeys $\langle V, \lambda  \rangle_\xM < +\infty$.
		\paragraph{Step 2. Continuous-time.}
		When $t > 0$, the measure $\mu$ in \eqref{equation:resstep1sm} is stationary for~$\mathfrak{P}_t^*$, as well. Indeed, given $t \in [0,T_0]$, the estimate \eqref{equation:c1iebrmo} yields
		\begin{equation}\label{equation:dis-cont}
			\begin{aligned}
				\langle V,\mathfrak{P}_t^*\nu \rangle_\xM & \leq q^t \langle V, \nu   \rangle_\xM + C\\
				  & \leq \langle V, \nu \rangle_\xM + C\langle V, \nu \rangle_\xM \\ 
                    & =  (C+1) \langle V, \nu \rangle_\xM <  +\infty
			\end{aligned}
		\end{equation}
		for any $\nu \in \mathcal{P}(\xM)$ with $\langle V,\nu   \rangle_\xM<+\infty$, where we used that $V\ge1$ in the second line. Thus, we can write $t=kT_0+s$, where $k \geq 0$ is an integer and~$s\in [0,T_0)$, followed by employing \eqref{equation:dis-cont} in order to conclude that~$\langle V, \mathfrak{P}_s^*\mu \rangle_\xM<+\infty$.
		Using \eqref{equation:resstep1sm} with~$\lambda = \mathfrak{P}_s^*\mu$, one can infer $\mathfrak{P}_s^*\mu=\mu$, which implies $\mathfrak{P}_t^*\mu=\mu$ for arbitrary $t \geq 0$.
		
		Finally, by decomposing again $t = kT_0 + s$, then taking any $\nu \in \mathcal{P}(\xM)$ such that $\langle V, \nu \rangle_\xM<+\infty$, 
		applying \eqref{equation:resstep1sm} with $\lambda \coloneq \mathfrak{P}_s^*\nu$, and utilizing the estimate \eqref{equation:dis-cont}, one arrives at
		\begin{align*}
			\|\mathfrak{P}_t^* \nu -\mu\|_{\operatorname{var}} & =\|\mathfrak{P}_{kT_0}^*( \mathfrak{P}_s^*\nu)  -\mu\|_{\operatorname{var}}\\
			& \leq C\operatorname{e}^{-\gamma k} \langle V, \mathfrak{P}_s^*\nu  \rangle_\xM\\
			& \leq (C+1)\operatorname{e}^{-\gamma t}  \langle V,  \nu  \rangle_\xM < + \infty
		\end{align*}
		for all $t \geq 0$.
	\end{proof}

	\subsection{Controllability of deterministic elasto-plasticity} \label{S:2.2}
	In this section, several controllability properties of deterministic dynamical systems associated with \eqref{equation:Application1_Model} are collected. We begin with showing the exact controllability of \eqref{equation:Application1_Model} in arbitrary time~$T>0$ to any target state $x_T \in \mathbb{R}_* \times (0,1)$. The resolving operator at time $t > 0$ for the system~\eqref{equation:Application1_Model} with control $\zeta = u$ is now denoted by
	\begin{gather*}
			S_t = (S_t^y, S_t^z)\colon \xM\times \xLtwo((0,t);\mathbb{R})\longrightarrow \xM, \quad (x_0, u) \mapsto x(t) = (y, z)(t).
	\end{gather*}
	\begin{prpstn}\label{proposition:Application1_exactcontrollability}
		Given any control time $T > 0$, initial data $x_0 = (y_0, z_0) \in \xM$, and target state $x_T = (y_T, z_T) \in \mathbb{R}_* \times (-1,1)$, there exists a control $u \in \xC([0,T];\mathbb{R})$ such that $S_T(x_0, u) = x_T$.
	\end{prpstn}
	\begin{proof} The desired control is obtained by gluing together the below-described building blocks in a continuous way (\cf~\Cref{example:elasto-plastic_controlled}). To simplify the presentation, but without loss of generality, it is assumed that $y_T < 0$; when $y_T > 0$, analogous constructions can be employed.

	\paragraph{Case 1. $y_0 = 0, z_0 \in (-1,1)$.} For any $\overline{u} \in \mathbb{R}$ with $f(x_0) + \overline{u} > 0$, there exists $\varepsilon_0 > 0$ so that $S^y_{\varepsilon}(x_0; \overline{u}) > 0$ and $|S^z_{\varepsilon}(x_0; \overline{u})| < 1$ for all $\varepsilon \in (0,\varepsilon_0)$.
		
	\paragraph{Case 2. $y_0 > 0, z_0 \in [-1,1)$.}
	There exists~$\varepsilon_0 > 0$ such that for any   $\varepsilon \in (0, \varepsilon_0)$ there is a control $\widetilde{u}\in \xC([0, \varepsilon];\mathbb{R})$ with
	\begin{equation*}\label{equation:ec_case1goal}
		\begin{gathered}
			S^y_{\varepsilon}(x_0; \widetilde{u}) > 0, \quad S^z_{\varepsilon}(x_0; \widetilde{u}) = 1, \quad | S^z_{t}(x_0; \widetilde{u})| \leq 1,  \quad t \in [0, \varepsilon].
		\end{gathered}
	\end{equation*} 
	Indeed, a desired controlled trajectory $(y,z)(t) = S_t(x_0; \widetilde{u})$ is given by
	\begin{equation*}
		y(t) \coloneq t a +y_0, \quad
		z(t) \coloneq \frac{t^2a}{2} + y_0 t +z_0, \quad 
		\widetilde{u}(t) \coloneq a-f(y(t),z(t)),
	\end{equation*}
	where
	\[
		a \coloneq \frac{2(1-z_0-y_0\varepsilon)}{\varepsilon^2}, \quad 0 < \varepsilon < \varepsilon_0 \coloneq \frac{1-z_0}{y_0}.
	\]

	\paragraph{Case 3. $y_0  \geq 0, z_0 = 1$.} For any $\widetilde{T} > 0$, there exists a control $\widetilde{u} \in \xC([0, \widetilde{T}];\mathbb{R})$ with $S_{\widetilde{T}}(x_0; \widetilde{u}) = (0,1)$. This can be seen by defining
	\begin{equation*}
		y(t) \coloneq  -\frac{y_0t}{\widetilde{T}}  + y_0, \quad z(t) \coloneq 1, \quad \widetilde{u}(t) \coloneq -\frac{y_0}{\widetilde{T}}-f(y(t),z(t)).
	\end{equation*}

	\paragraph{Case 4. $x_0 =(0,-1)$.} Let $\overline{u} \in \mathbb{R}$ with $f(x_0) + \overline{u} > 0$, and select $\varepsilon > 0$ so small that $S^y_\varepsilon(x_0;\overline{u}) > 0$. Together with the analysis of Case 3, for $\widetilde{T} > 0$, this provides a control $\widetilde{u} \in \xC([0, \widetilde{T}];\mathbb{R})$ such that~$S_{\widetilde{T}}(x_0; \widetilde{u}) = (0,1)$.	
		
	\paragraph{Case 5. $y_0 \neq 0, z_0 \in [-1,1]$.} For~$\widetilde{T} > 0$, the previous cases (and similar arguments when $y_0 < 0$) lead to a control $\widetilde{u} \in \xC([0,\widetilde{T}];\mathbb{R})$ such that $S_{\widetilde{T}}(x_0; \widetilde{u}) = (0,1)$. 
		
	\paragraph{Case 6. $x_0 =(0,1)$.} Given any $\widetilde{T} > 0$, take a monotonic $\varphi \in \xCinfty([0,\widetilde{T}];\mathbb{R}_-)$ with
	\begin{gather*}
		\varphi(t) = 0 \iff t = 0, \quad \varphi(\widetilde{T}) = y_T, \quad  \int_0^{\widetilde{T}}\varphi(s) \xdx{s} = z_T - 1.
	\end{gather*}
	Then, to achieve $(y,z)(t) = S_t(x_0; \widetilde{u})$ for $t\in [0,\widetilde{T}]$ and $S_{\widetilde{T}}(x_0; \widetilde{u})=x_T$, we choose
	\begin{equation*}
		y(t) \coloneq \varphi(t), \quad
		z(t) \coloneq \int_0^t \varphi(s) \xdx{s} + 1, \quad
		\widetilde{u}(t) \coloneq \dot y(t)-f(y(t),z(t)).
	\end{equation*}
	
	\end{proof}

	\begin{figure}[ht!]
		\centering
		\resizebox{0.75\textwidth}{!}{
			\begin{tikzpicture}
				\clip(0,1.2) rectangle (8,6.82);
				
				\draw[->, line width=0.2mm, dashed] (0,4) -- (8,4);
				\draw[->, line width=0.2mm, dashed] (4,0.5) -- (4,6.8);
				
				\draw[line width=0.4mm, color=black] (0.76,6.2) -- (8,6.2);
				\draw[line width=0.4mm, color=black] (0,1.8) -- (7,1.8);
				\fill[line width=0pt, color=black, pattern={mylines[size=3pt,line width=0.9pt,angle=72]},
				pattern color=black] (4,6.2) rectangle (8,6.3);
				\fill[line width=0pt, color=black, pattern={mylines[size=2.9pt,line width=0.9pt,angle=72]},
				pattern color=black] (4,1.7) rectangle (0,1.8);
				
				\def \Samples {91};
				\def \Epsilon {1/2};
				\def \Y {3.3}; 
				\def \Z {4.2}; 
				\def \YShift {4};
				\def \ZShift {3.15};
				
				\foreach \i in {1,2,...,\Samples} {
					\coordinate(A\i) at ({\Y - ((2*(1+(\Z - \ZShift) + (\Y-\YShift)*\Epsilon)))*\Epsilon*(\i/(\Samples-1)-1/(\Samples-1))}, {\Z + (\Y-\YShift)*(\i/(\Samples-1)-1/(\Samples-1)) - ((1 + (\Z - \ZShift) + (\Y-\YShift)*\Epsilon))*(\i/(\Samples-1)-1/(\Samples-1))^2});
				}
				
				\draw [line width=0.5pt, color=black, postaction={decorate,decoration={markings,
						mark=at position 0.6 with {\arrow[line width=2.8pt, color=FireBrick!80]{stealth}}}}] plot[domain=1:\Samples, samples=\Samples] (A\x);
					
				\draw [line width=0.5pt, color=black, postaction={decorate,decoration={post length=2pt,
						pre length=8pt,
						markings,
						mark=at position 0.5 with {\arrow[line width=2.8pt, color=FireBrick!80]{stealth}}}}]  plot[smooth, tension=0.2] coordinates {(A91) (4,1.8)};
				
				\draw [line width=0.5pt, color=black, postaction={decorate,decoration={post length=2pt,
						pre length=8pt,
						markings,
						mark=at position 0.7 with {\arrow[line width=2.8pt, color=FireBrick!80]{stealth}}}}]  plot[smooth, tension=1] coordinates {(4,1.8) (4.3,1.98) (4.7,2)};
					
				\draw [line width=0.5pt, color=black, postaction={decorate,decoration={post length=2pt,
						pre length=8pt,
						markings,
						mark=at position 0.35 with {\arrow[line width=2.8pt, color=FireBrick!80]{stealth}},
						mark=at position 0.8 with {\arrow[line width=2.8pt, color=FireBrick!80]{stealth}}}}]  plot[smooth, tension=1] coordinates {(4.7,2) (4,6.2)};
				
				\foreach \i in {1,2,...,\Samples} {
					\coordinate(B\i) at ({4 - (1.3)*(\i/(\Samples-1)-1/(\Samples-1))}, {6.2-(0.5)*(1.3)*(1.5)*(\i/(\Samples-1)-1/(\Samples-1))^2});
				}
				
				\draw [line width=0.5pt, color=black, postaction={decorate,decoration={markings,
						mark=at position 0.65 with {\arrow[line width=2.8pt, color=FireBrick!80]{stealth}}}}] plot[domain=1:91, samples=91] (B\x);

				\fill[color=black, fill=FireBrick!80] plot[smooth cycle] (B91) circle (0.08);
				\fill[color=black, fill=FireBrick!80] plot[smooth cycle] (A1) circle (0.08);
				
				\coordinate[label=below:{\footnotesize$y$}] (y_axis) at (7.9,4);
				\coordinate[label=below:{\footnotesize$z$}] (z_axis) at (3.82,6.89);
				\coordinate[label=right:\footnotesize{$x_0$}] (Start) at (A1);
				\coordinate[label=left:{\footnotesize$x_T$}] (End) at (B91);
				\coordinate[label=below:{\footnotesize$z=1$}] (Bg_Rectangle_1) at (0.35,6.42);
				\coordinate[label=below:{\footnotesize$z=-1$}] (Bg_Rectangle_2) at (7.55,2.02);
			\end{tikzpicture}
		}
		\caption{A schematic sketch of a controlled trajectory~$x = (y, z)$. An initial state~$x_0$ which is, for instance, situated in the second quadrant of the~$yz$-plane, is connected with a prescribed target state~$x_T$ lying northwest of~$x_0$ (\cf~\Cref{example:elasto-plastic_controlled}). The reachable set in the plastic phase is indicated by the shaded parts of the lines~$z=\pm1$. Red arrows emphasize the orientation of~$t \mapsto (y, z)(t)$.}
		\label{Figure:ExampleTrajectory}
	\end{figure}
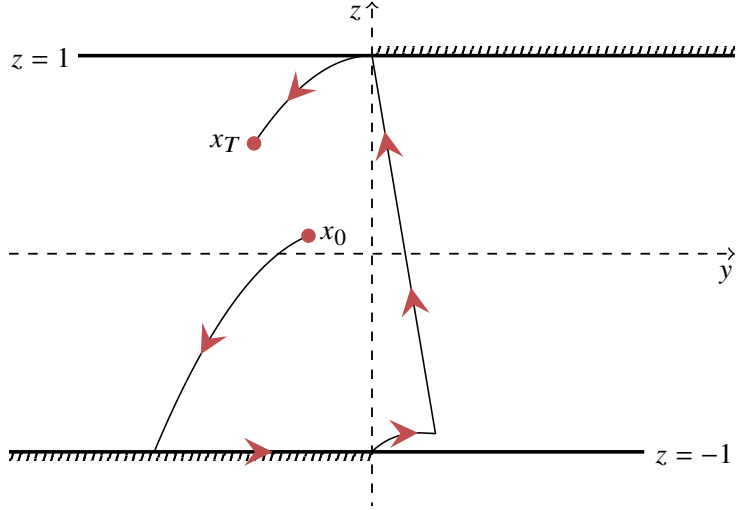
	
	\begin{xmpl}\label{example:elasto-plastic_controlled}
		To illustrate the gluing argument implicitly used in the proof of \Cref{proposition:Application1_exactcontrollability}, let $y_0 < 0$, $z_0 \in (-1,1]$, $y_T < y_0$, and $z_T \in (z_0,1)$. First, in order to connect $x_0 = (y_0, z_0)$ with the line $\{z = -1\}$, one can choose a sufficiently small number $\varepsilon_0 > 0$ and replace $(y,z,\widetilde{u},a)$ in Case~2 by
		\begin{equation*}
			y_1(t) \coloneq y_0 - t a, \quad
			z_1(t) \coloneq -\frac{t^2 a}{2} + y_0 t + z_0, \quad 
			u_1(t) \coloneq a-f(y(t),z(t)),
		\end{equation*}
		where $a = a(\varepsilon) > 0$ is for $\varepsilon \in (0, \varepsilon_0)$ fixed via
		\[
		a \coloneq \frac{2(1+z_0+y_0\varepsilon)}{\varepsilon^2}.
		\]
		Next, we denote a control similar to that from Case~4 with $\widetilde{T} \coloneq T/2$ as $u_2$, the control given by Case~6 for $\widetilde{T} = T/2-\varepsilon$ as $u_3$, and eventually define (\cf~\Cref{Figure:ExampleTrajectory})
		\[
		u(t) \coloneq \begin{cases}
			u_1(t) & \mbox{ if } t \in [0, \varepsilon],\\
			u_2(t) & \mbox{ if } t \in (\varepsilon, T/2+\varepsilon],\\
			u_3(t) & \mbox{ if } t \in (T/2+\varepsilon, T].
		\end{cases}
		\]
	\end{xmpl}
	
	The solid controllability of \eqref{equation:Application1_Model}, as stated in \Cref{condition:C3}, shall be demonstrated next. However, let us first recall a version of an auxiliary lemma which has been established in \cite[Proof of Theorem 2.1, Step 1]{Sh-17}.
	
	\begin{lmm}\label{lemma:continuous_excat_ctrl_implies_solidctrl}
		Let $\xX$ denote a compact metric space and $\xY$ be a separable Banach space.  Moreover, assume the existence of
		\begin{itemize}
			\item two balls $\xG \subset \xY$ and $\varnothing \neq \xF \subset \xX$,
			\item two functions $H \in \xC(\xG; \xX)$ and $h \in \xC(\xF; \xG)$ satisfying $H(h(x)) = x$ for all~$x\in \xF$.
		\end{itemize}
		There are a number $\varepsilon >0$, a nonempty ball $\xB \subset \xX$, and a compact set~$\mathcal{K} \subset \xG$
		such that for each function $\Phi \in \xC(\mathcal{K};\xX)$ that obeys the inequality 
		\[
			\sup_{\zeta \in \mathcal{K}} d_X(\Phi(\zeta),H(\zeta)) \leq \varepsilon
		\]
		one has the inclusion $\xB \subset \Phi(\mathcal{K})$.
	\end{lmm}

	\begin{rmrk}
		The mapping $H$ in \Cref{lemma:continuous_excat_ctrl_implies_solidctrl} shall play the role of $S_{T_0}(x_0, \cdot)$, while the continuous function $h$ produces a suitable control for each admissible target state. In other words, the \enquote{continuous exact controllability} of $H$ implies its \enquote{solid controllability}. Thus, verifying \Cref{condition:C3} essentially reduces to checking the hypotheses of \Cref{lemma:continuous_excat_ctrl_implies_solidctrl}.
	\end{rmrk}

	Recall that $p$ is the point from \Cref{assumption:Application1_force}, while $T_0 \in (0, 1]$ denotes the corresponding reference time fixed in \Cref{subsection:TheElastoPlasticModel}. In particular, if~$p \in \{0\}\times(-1,1)$, then~$T_0=1$.
	
	\begin{prpstn}\label{proposition:ressc3}
		The resolving operator at $t = T_0$ of \eqref{equation:Application1_Model} with control $\zeta = u$, namely
		\begin{gather*}
			S = S_{T_0} \colon \xM\times \xE \longrightarrow \xM, \quad (x_0, u) \mapsto x(T_0) = (y, z)(T_0),
		\end{gather*}
		satisfies \Cref{condition:C3}.
	\end{prpstn}
	\begin{proof}
		Due to \Cref{assumption:Application1_force}, there exists an open neighborhood~$\widetilde{\xB}_{\xM} \subset \xM \setminus \partial \xM$ of a point $p \in \xM$ for which the mapping $f\colon \widetilde{\xB}_{\xM} \longrightarrow \mathbb{R}$ in \eqref{equation:Application1_Model} is smooth. 
		\paragraph{Step 1. Regularity.} Due to the choice of $T_0$ at the beginning of \Cref{subsection:dn} (see in particular \eqref{equation:characterizationT0} if $y_p \neq 0$), one can take $\delta_0 > 0$ so small that 
		\begin{itemize}
			\item $\xB_{\xM}(p, \delta) \subset \widetilde{\xB}_{\xM}$,
			\item the restriction of $S = S_{T_0}$ to the set $\xB_{\xM}(p, \delta)\times \xB_{\xE}(0,\delta)$ constitutes a smooth mapping $\xB_{\xM}(p, \delta)\times \xB_{\xE}(0,\delta) \longrightarrow \widetilde{\xB}_{\xM}$
		\end{itemize}
		for any $\delta \in (0,\delta_0]$. From now on, such a number~$\delta\in(0,\delta_0]$ is fixed. 
		
		\paragraph{Step 2. Solid controllability.}
		The property (ii) of \Cref{condition:C3} will follow from an application of \Cref{lemma:continuous_excat_ctrl_implies_solidctrl} with $H = S(p;\cdot)$. Hereto, given any element~$u \in \xB_{\xE}(0,\delta)$, we denote by
		\[
			(y(t), z(t)) = S_t(p;u)
		\]
		the solution at time $t\in [0,T_0]$ of the initial value problem
		\begin{equation}\label{equation:systemproofsolid}
			\dot y = f(y,z) + u, \quad \dot z = y, \quad (y, z)(0) = p. 
		\end{equation}
		To establish the existence of a continuous state-to-control mapping, we fix any reference point $\overline{u} \in \xB_{\xE}(0,\delta)$ and then linearize \eqref{equation:systemproofsolid} about the trajectory
		\[
			(\overline{y}, \overline{z})(t) = S_t(p; \overline{u}).
		\]
		More precisely, given any control~$V\in \xE$, we consider for $t\in [0,T_0]$ the linear problem with vanishing initial states
		\begin{equation}\label{equation:linearizationproofsolid}
			\dot Y = \partial_y f (\overline{y}, \overline{z}) Y + \partial_z f (\overline{y}, \overline{z}) Z + V, \quad
			\dot Z = Y, \quad  (Y, Z)(0) = (0, 0),
		\end{equation}
		and denote by $\{\mathscr{R}_t\}_{t\in [0,T_0]}$ the associated resolving family. This means that, given any~$t \in [0, T_0]$, the linear operator~$\mathscr{R}_t$ maps each $V \in \xE$ to $(Y,Z)(t)$, where $(Y,Z)$ solves \eqref{equation:linearizationproofsolid} with $(Y, Z)(0) = (0, 0)$. In symbols,
		\[
			\mathscr{R}_t(\overline{y}, \overline{z})\colon \xE\longrightarrow \mathbb{R}^2, \quad V \mapsto (Y, Z)(t).
		\]
		In particular, one can show that the map
		\[
			\mathscr{R}_{T_0}(\overline{y}, \overline{z})\colon \xE \longrightarrow \mathbb{R}^2
		\]
		is onto. To see this, we take any target state~$(Y_1, Z_1)\in \mathbb{R}^2$ and choose a smooth function~$\varphi\colon [0,T_0]\longrightarrow \mathbb{R}$ with
		\[
			\varphi(0)=0, \quad \varphi(T_0)= Y_1, \quad  \int_0^{T_0} \varphi(s) \, \xdx{s} = Z_1.
		\]
		Then, owing to the well-posedness of \eqref{equation:linearizationproofsolid}, the profiles 
		\begin{equation*}
			Y(t) \coloneq  \varphi(t), \quad Z(t) \coloneq \int_0^t \varphi(s) \, \xdx{s}
		\end{equation*}
		and control
		\begin{equation}\label{equation:pctrlc}
				V(t) \coloneq  \dot{Y}(t) - \partial_y f(\overline{y}(t), \overline{z}(t))Y(t) - \partial_z f(\overline{y}(t), \overline{z}(t))Z(t)
		\end{equation}
		satisfy
		\[
			(Y(t), Z(t)) = \mathscr{R}_t(\overline{y}, \overline{z}) V, \quad (Y(T_0), Z(T_0)) = (Y_1, Z_1), \quad t \in [0, T_0].
		\]
		Moreover, the inverse function theorem (\cf~\cite[Part 2, Section 3.1.1]{Coron-07}) provides a closed ball $\xF \subset \xM \setminus \partial \xM$ and a continuous mapping $s\colon \xF\longrightarrow \xB_{\xE}(0,\delta)$ such that
		\[
			S(p; s(x)) = x, \quad x\in \xF.
		\]
		Finally, that \Cref{condition:C3} is verified can be seen by an application of \Cref{lemma:continuous_excat_ctrl_implies_solidctrl} with~$\xG = \xB_{\xE}(0,\delta)$, $H(\cdot) = S(p;\cdot)$, and $h = s$.
	\end{proof}

	\appendix
	\gdef\thesection{\Alph{section}}
	\makeatletter
	\renewcommand\@seccntformat[1]{Appendix \csname the#1\endcsname.\hspace{0.5em}}
	\makeatother
	
	\section{Exponential recurrence}\label{section:appendix_exponentialrecurrence}
	In what follows, we suppose that the map~$S$ in~\eqref{equation:generalRDS} satisfies Conditions~\Rref{condition:C1} and~\Rref{condition:C2}; the notations from these conditions are employed below. Let~$\ell$ be the law of~$\eta_1$ from \Cref{subsection:fourconditions}, and, given any~$\delta > 0$, denote the first hitting time of~$\xB_\xM(p, \delta)$ as
	\[
		\tau(p,\delta) \coloneq \tau_{\delta} \coloneq \min\left\{k  \geq 1 \, | \, x_k\in \xB_\xM(p,  \delta)\right\},
	\]
	where $\min\varnothing = +\infty$ by convention. 
	\begin{lmm}\label{lemma:recurrence} 
	Assuming that $\operatorname{supp}\ell=\xE$, one has $\mathbb{P}_x\{\tau_{\delta} < +\infty\}=1$. Moreover, there are numbers $\varkappa>0$ and $C>0$ such~that 
	\[	
		\mathbb{E}_x\operatorname{e}^{\varkappa \tau_{\delta}} \leq  C V(x)
	\]
	for all $x\in \xM$.	
	\end{lmm}
	\begin{proof}
		Given any $R>0$, we denote the set $\xK_R \coloneq \{x\in \xM\,\,|\,\, V(x)\leq R\}$, which is compact due to \Cref{condition:C1}, and consider its first hitting time 
		\[
				T_R\coloneq \min\{k \geq 1\,\,|\,\, x_k\in \xK_R\}. 	
		\]
		Then, by literally repeating the stopping times argument from \cite[Section~3.3.2]{KS-12}, the proof of \Cref{lemma:recurrence} reduces to showing the following two properties.
		\begin{itemize}
			\item[($\alpha$)] There are $R, \varkappa_1, C > 0$ such that $\mathbb{E}_x\operatorname{e}^{\varkappa_1T_R} \leq C V(x)$ for all $x\in \xM$.
			\item[($\beta$)] There exist $n \in \mathbb{N}$ and $r\in (0,1)$ with $	\mathbb{P}_x\{ x_n\in \xB_\xM(p, {\delta})\} \geq r$ for all $x \in \xK_R$.
		\end{itemize}
		
		\paragraph{Step 1. Verification of ($\alpha$).}
		Thanks to the Lyapunov type inequality \eqref{equation:abstractLyapunov} ensured by \Cref{condition:C1}, it follows that
		\begin{equation}\label{equation:apptmp1}
			\mathbb{E}_x V(x_j) \leq q \mathbb{E}_x V(x_{j-1})+A, \quad j \geq 1 
		\end{equation} 
		for some $q \in (0,1)$ and $A > 0$. Let us choose any   $\gamma\in (1,1/q)$, multiply~\eqref{equation:apptmp1} by~$\gamma^j$, and sum up the resulting inequalities for $j=1,\ldots, k$; this yields
		\begin{equation}\label{equation:apptmp2}
			\sum_{j=1}^k  \gamma^j \mathbb{E}_x V(x_j)\leq q \gamma \sum_{j=0}^{k-1}  \gamma^j \mathbb{E}_x V(x_j)+A\sum_{j=1}^k \gamma^j, \quad k \geq 1. 
		\end{equation}
		Using the fact that $V \geq 1$,	the inequality \eqref{equation:apptmp2} can be rewritten as   
		\begin{equation}\label{equation:apptmp3}
			\mathbb{E}_x \sum_{j=1}^k \gamma^j \left( (1-q\gamma) V(x_j) -A\right) \leq q \gamma V(x), \quad k \geq 1.
		\end{equation}
		By choosing $R > A/(1-q\gamma)$, resorting to the Chebyshev inequality, and using the estimate in \eqref{equation:apptmp3}, one then obtains for any $k \geq 1$ that
		\begin{equation}\label{equation:apptmp4}
			\begin{aligned}
				\mathbb{P}_x\{T_R>k\} & \leq \mathbb{P}_x\{V(x_j)> R, \,\,j=1,\ldots,k\}\\
				& \leq \frac{q(\gamma-1)V(x)}{((1-q\gamma)R-A)(\gamma^k-1)} \\
				& \leq C\gamma^{-k} V(x).
			\end{aligned}
		\end{equation}
		Thus, resorting to the Borel--Cantelli lemma, one finds that~$\mathbb{P}_x\{T_R<+\infty\} = 1$ for all~$x\in \xM$. Furthermore, from \eqref{equation:apptmp4} it can be derived that
		\begin{align*}
			\mathbb{E}_x\operatorname{e}^{\varkappa_1T_R} &\leq 1+ \operatorname{e}^{\varkappa_1}+\sum_{k=1}^{\infty}  \mathbb{E}_x \left\{\operatorname{e}^{\varkappa_1T_R} \mathbb{I}_{\{k < T_R \leq k+1\}}\right\}\\&\leq   1+ \operatorname{e}^{\varkappa_1}+\sum_{k=1}^{\infty} \operatorname{e}^{\varkappa_1(k+1)} \mathbb{P}_x \left\{ T_R>k \right\} \\&\leq   1+ \operatorname{e}^{\varkappa_1} + C V(x) \sum_{k=1}^{\infty} \operatorname{e}^{\varkappa_1(k+1)} \gamma^{-k}. 
		\end{align*}
		Therefore, by taking $\varkappa_1 < \log \gamma$, we arrive at the desired inequality.
		
		\paragraph{Step 2. Verification of ($\beta$).} 
		Let us fix $R'>0$ so large that $\xK_R\subset \xB_\xM(p, R')$. Then, by \Cref{condition:C2}, there exist $n = n(R', \delta) \in \mathbb{N}$ and vectors $\zeta_1, \dots, \zeta_n \in \xE$ such that
		\begin{equation}\label{equation:apptmp5}
			S_n(x;\zeta_1, \ldots,\zeta_n)\in  \xB_\xM(p, \delta) 
		\end{equation}
		holds for all $x\in \xB_\xM(p, R')$. Furthermore, recalling that $\eta_1, \ldots,\eta_n$ are i.i.d, and their law obeys $\operatorname{supp}\ell = \xE$ by assumption, one has
		\begin{equation}\label{equation:apptmp6}
			\begin{aligned}
				r_{\varepsilon} & \coloneq \mathbb{P}\{\eta_1 \in \xB_\xE(\zeta_1,\varepsilon), \ldots, \eta_n \in \xB_\xE(\zeta_n,\varepsilon)\} \\
				& \, \, = \prod_{j=1}^n\mathbb{P}\left\{\eta_j \in \xB_\xE(\zeta_j,\varepsilon)\right\}>0
			\end{aligned}
		\end{equation}
		for arbitrary $\varepsilon > 0$. Therefore, since the mapping $S_n\colon \xM\times \xE^n\longrightarrow \xM$ is continuous, one can conclude with the help of \eqref{equation:apptmp5} and \eqref{equation:apptmp6} the estimate
		\[
			\mathbb{P}_y\{ x_n\in  \xB_\xM(p, \delta)\} \geq  r_{\varepsilon} 
		\]
		for all $y \in \xB_\xM(x,\varepsilon)$ and a sufficiently small choice of $\varepsilon > 0$. The property ($\beta$) is then deduced from the compactness of $\xK_R$.
	\end{proof}
	
	Similar exponential recurrence property holds for the coupling process $\{(\widetilde x_k,\widetilde x_k')\}$ constructed in \Cref{subsubsection:cs}. 
	\begin{lmm}\label{lemma:recurrence2}
		Under the conditions of \Cref{theorem:main}, given any $\delta>0$ and the hitting time
		\[
			\widetilde{\tau}_{\delta} \coloneq \min\{k \geq 1 \,| \,  \widetilde x_k,\widetilde x_k'\in \xB_\xM(p, \delta)\}, 
		\]
		there are constants $\varkappa > 0$ and $C>0$ such that
		\[
			\mathbb{E}_{(x,x')}\operatorname{e}^{\varkappa \widetilde{\tau}_{\delta}} \leq  C(V(x)+V(x'))
		\]
		 for all $x, x'\in \xM$.
	\end{lmm}
  	\begin{proof}[Sketch of the proof]
			The proof repeats the arguments of the previous lemma in the case of the process~$\{(\widetilde x_k,\widetilde x_k)\}_{k \in \mathbb{N}}$. For any $R>0$, let 
			$$
			\xK_R \coloneq \{(x,x')\in \xM\times\xM \,\,|\,\, V(x)+V(x')\leq R\},
			$$   and consider the first hitting time 
		\[
				T_R\coloneq \min\{k \geq 1\,\,|\,\, (\widetilde x_k,\widetilde x_k')\in \xK_R\}. 	
		\]
		The result will be proved if the following properties are established:
		\begin{itemize}
			\item[($\widetilde{\alpha}$)] there are $R, \varkappa_1, C > 0$ such that $\mathbb{E}_{(x,x')}\operatorname{e}^{\varkappa_1T_R} \leq C (V(x)+V(x'))$ for all $x,x'\in \xM$;
			\item[($\widetilde{\beta}$)] there exist $n \in \mathbb{N}$ and $r\in (0,1)$ with $	\mathbb{P}_{(x,x')}\{  \widetilde x_n,\widetilde x_n\in \xB_\xM(p, {\delta})\} \geq r$ for all $(x,x') \in \xK_R$.
		\end{itemize}
			 Property ($\widetilde{\alpha}$) is checked by repeating the arguments in Step~1 of \Cref{lemma:recurrence}'s proof, using that  $V(x)+V(x')$ is a Lyapunov function for the process $\{(\widetilde{x}_k,\widetilde{x}_k')\}_{k \in \mathbb{N}}$. Condition~\Rref{condition:C4} garantees that $\operatorname{supp}\ell=\xE$. Then property ($\widetilde{\beta}$) is obtained as in Step~2 of \Cref{lemma:recurrence}, using the construction of $\{(\widetilde{x}_k, \widetilde{x}_k')\}_{k \in \mathbb{N}}$.
	\end{proof}
	
	\section{Measure transformation theorem}\label{section:appendix_measuretransformation}
	Given a compact metric space $\xX$, a separable Banach space $\xE$, and a smooth Riemannian manifold $\xM$ as in \Cref{subsection:fourconditions}, let
	\[
		F\colon \xX\times \xE \longrightarrow \xM, \quad (x, \eta) \mapsto F(x, \eta)
	\]
	be a continuous mapping. The following theorem is a consequence of   Theorem~2.4 in~\cite{Sh-07}; see also~\cite[Chapter~10]{Bog-10} for related results. 
	\begin{thrm}\label{theorem:measuretransformation}
		Let $p\in \xX$, $\zeta \in \xE$, and~$\delta>0$ be chosen such that the following conditions are satisfied.
		\begin{enumerate}
			\item $F(p, \zeta)\in \xM\setminus \partial \xM$,
			\item $F(x,\,\cdot\,)\colon {\xB}_\xE(\zeta, \delta)\longrightarrow \xM$ is Fr\'echet differentiable for any $x\in {\xB}_\xX(p, \delta)$,
			\item $D_{\eta} F$ is continuous on~${\xB}_\xX(p, \delta)\times {\xB}_\xE(\zeta, \delta)$,
			\item the image of the linear mapping $(D_{\eta} F)(p, \zeta)$ has full rank,
			\item $\ell \in\mathcal{P}(\xE)$ is a measure obeying the properties in \Cref{condition:C4}.
		\end{enumerate}
		Then there is a number~$\widehat{\delta} > 0$ and a continuous function $\psi\colon \xB_\xX(p,\widehat{\delta})\times \xM\longrightarrow\mathbb{R}_+$ satisfying for $x\in \xB_\xX(p,\widehat{\delta})$ the relations
		\begin{equation*}
			\psi(p, F(p, \zeta)) > 0, \quad 
			\left(F(x,\,\cdot\,)_*\ell\right)(\xdx{y})  \geq \psi(x, y)\,\operatorname{vol}_\xM(\xdx{y}), 
		\end{equation*}
		where $\operatorname{vol}_\xM(\,\cdot\,)$ denotes the Riemannian measure on~$\xM$ and $F(x,\,\cdot\,)_*\ell$ stands for the image of $\ell$ under the mapping $\eta\mapsto F(x,\eta)$.
	\end{thrm}

	\section*{Acknowledgements}
	L.M. is thankful for support through NSFC Grant 
        No. 12271364 and GRF Grant No. 11302823.
 
	\addcontentsline{toc}{section}{References}
	\bibliography{references}
	\bibliographystyle{abbrv}
\end{document}